\documentclass[leqno,10.5pt]{article} 
\usepackage{amsmath, amssymb} 
\usepackage{amsthm} 
\usepackage{amscd} 
\usepackage{amsmath, amssymb}
\usepackage{amsthm} 
\usepackage{amscd} 
\theoremstyle{plain} 
\newtheorem{theorem}{\indent\sc Theorem}[section]
\newtheorem{lemma}[theorem]{\indent\sc Lemma}
\newtheorem{corollary}[theorem]{\indent\sc Corollary}
\newtheorem{proposition}[theorem]{\indent\sc Proposition}

\theoremstyle{definition} 
\newtheorem{definition}[theorem]{\indent\sc Definition}
\newtheorem{remark}[theorem]{\indent\sc Remark}

\newcommand{\C}{\mathbb{C}}
\newcommand{\R}{\mathbb{R}}
\newcommand{\Q}{\mathbb{Q}}
\newcommand{\Z}{\mathbb{Z}}
\newcommand{\N}{\mathbb{N}}

\title{Linear independence of values of logarithms revisited}
\author{\textsc{Makoto Kawashima}}
\date{} 
\begin{document}

\maketitle

\tableofcontents


\begin{abstract}
Let $m\ge 2$ be an integer, $K$ an algebraic number field and $\alpha\in K\setminus \{0,-1\}$ with sufficiently small absolute value.
In this article, we provide a new lower bound for linear form in $1,{\rm{log}}(1+\alpha),\ldots,{\rm{log}}^{m-1}(1+\alpha)$ with algebraic integer coefficients in both complex and $p$-adic cases (see Theorem $\ref{power of log indep}$ and Theorem $\ref{p power of log indep}$). Especially, in the complex case, our result is a refinement of the result of Nesterenko-Waldschmidt on the lower bound of linear form in certain values of power of logarithms. The main integrant is based on Hermite-Mahler's Pad\'{e} approximation of exponential and logarithm functions.
\end{abstract}
\section{Introduction}
Let $\theta$ be a transcendental complex number. The function $\Phi:\N\times \R_{>0}\longrightarrow \R_{>0}$ is a transcendental measure for $\theta$ if, for any sufficiently large positive integer $m$, any sufficiently large positive real number $H$ and any nonzero polynomial $P(z)\in \Z[z]$ with ${\rm{deg}}(P)\le m$ and ${\rm{H}}(P)\le H$, we have
$${\rm{exp}}(-\Phi(m,H)) \le |P(\theta)|.$$ 
Let $\alpha$ be an algebraic number different from $0$ and $1$. Then the complex number ${\rm{log}}(\alpha)$ is transcendental. 
A great deal of work has already been done on finding transcendence measures for the values of logarithms.
For example, Mahler \cite{M1},\cite{M2},\cite{M3}, Gel'fond \cite{G}, Feldman \cite{F}, Cijsouw \cite{C}, Reyssat \cite{R} and Waldschmidt \cite{W1}. 
In \cite{N-W}, Nesterenko-Waldschmidt gave the following transcendence measure of values of logarithm.
\begin{theorem}  $($\rm{\cite[Theorem $6.$ $1)$]{N-W}}$)$ \label{N-W}
Let $\alpha$ be an algebraic number, $\alpha\neq 0,1$. 
Then there exists a effectively computable positive number $C=C(\alpha)$, depending only on $\alpha$ and the determination of the logarithm of $\alpha$ such that 
if $P(z)\in \Z[z] \setminus\{0\}$, ${\rm{deg}}P\le m, {\rm{L}}(P)\le L$, then
\begin{align} \label{N-W lower bound}
|P({\rm{log}}(\alpha))|\ge {\rm{exp}}\left(-Cm^2({\rm{log}}(L)+m{\rm{log}}(m))(1+{\rm{log}}(m))^{-1}\right)
\end{align}
where ${\rm{L}}(P)=\sum_{i=0}^m|a_i|$ if $P(z)=\sum_{i=0}^m a_iz^{i}$.
\end{theorem}
The purpose of the present article is to give an improvement of Theorem $\ref{N-W}$ for algebraic numbers $\alpha$ which are sufficiently close to $1$ and give a $p$-adic version of the result.
The main integrant of the proof of these results is Hermite-Pad\'{e} approximation of exponential and logarithm functions.
\section{Notations and main results}
We collect some notations which we use throughout this article. 
For a prime number $p$, we denote the $p$-adic number field by $\Q_p$, the $p$-adic completion of a fixed algebraic closure of $\Q_p$ by $\C_p$
and the normalized $p$-adic valuation on $\C_p$ by $$|\cdot|_p:\C_p\longrightarrow \R_{\ge0}, \ |p|=p^{-1}.$$ 
We fix an algebraic closure of $\Q$ and denote it by $\overline{\Q}$. We define the denominator function by 
$${\rm{den}}:\overline{\Q}\longrightarrow \N, \ \alpha\mapsto \min\{n\in\N\mid n\alpha \ \text{is an algebraic integer}\}.$$
We fix embeddings $\sigma:\overline{\Q}\hookrightarrow \C$ and $\sigma_{p}:\overline{\Q}\hookrightarrow \C_p$. 
For an algebraic number field $K$, we consider $K$ as a subfield of $\overline{\Q}$ and denote the ring of integers of $K$ by $\mathcal{O}_K$.
For $\alpha\in K$, we denote $\sigma(\alpha)=\alpha$, $\sigma_{p}(\alpha)=\alpha$ and the set of conjugates of $\alpha$ by $\{\alpha^{(k)}\}_{1\le k \le [K:\Q]}$ with $\alpha^{(1)}=\alpha$ and $\alpha^{(2)}$ is the complex conjugate of $\alpha$ if $\sigma(K)\not\subset \R$. 
We denote the set of places of $K$ (resp. infinite places, finite places) by $M_{K}$ (resp. $M^{\infty}_{K}$, $M^{f}_{K}$).
For $v\in M_{K}$, we denote the completion of $K$ with respect to $v$ by $K_v$.
For $v\in M_{K}$, we define the normalized absolute value $| \cdot |_v$ as follows:
\begin{align*}
&|p|_v:=p^{-\tfrac{[K_v:\Q_p]}{[K:\Q]}} \ \text{if} \ v\in M^{f}_{K} \ \text{and} \ v|p,\\
&|x|_v:=|\sigma_v x|^{\tfrac{[K_v:\R]}{[K:\Q]}} \ \text{if} \ v\in M^{\infty}_{K},
\end{align*}
where $\sigma_v$ is the embedding $K\hookrightarrow \C$ corresponding to $v$. 
Then we have the product formula 
\begin{align*} 
\prod_{v\in M_{K}} |\xi|_v=1 \ \text{for} \ \xi \in K\setminus\{0\}.
\end{align*}

\

Let $m$ be a natural number and $\boldsymbol{\beta}:=(\beta_0,\ldots,\beta_m) \in K^{m+1} \setminus\{\bold{0}\}$. 
We define the absolute height of 
$\boldsymbol{\beta}$ by
\begin{align*}
&\mathrm{H}(\boldsymbol{\beta}):=\prod_{v\in M_{K}} \max\{1, |\beta_0|_v,\ldots,|\beta_m|_v\}.
\end{align*}
Note that, for $\boldsymbol{\beta}=(\beta_0,\ldots,\beta_{m})\in \mathcal{O}^{m+1}_K \setminus\{\bold{0}\}$, we have $\mathrm{H}(\boldsymbol{\beta})=\prod_{v\in M^{\infty}_{K}} \max\{1, |\beta_0|_v,\ldots, |\beta_m|_v\}$ and 
\begin{align} \label{important equal}
\prod_{k=1}^{[K:\Q]} \max\{1, |\beta^{(k)}_0|,\ldots, |\beta^{(k)}_m|\}=\mathrm{H}(\boldsymbol{\beta})^{[K:\Q]}.
\end{align}

\

Let ${\rm{log}}:\C\setminus\R_{\le 0}\longrightarrow \C$ be the principal value logarithm function and ${\rm{log}}_p:\C_p\setminus\{0\} \longrightarrow \C_p$ the $p$-adic logarithm function, i.e. 
${\rm{log}}_p$ is a $p$-adic locally analytic function satisfying the following conditions:
\begin{align*}
&({\rm{i}}) \ {\rm{log}}_p(1+z)=\sum_{k=1}^{\infty}\dfrac{(-1)^{k+1}z^k}{k} \ \text{if} \ |z|_p<1,\\
&({\rm{ii}}) \ {\rm{log}}_p(xy)={\rm{log}}_p(x)+{\rm{log}}_p(y) \ \text{for} \ x,y\in \C_p\setminus\{0\},\\
&({\rm{iii}}) \ {\rm{log}}_p(p)=0.
\end{align*}
Under the above notations, we shall prove the following results.
\begin{theorem}\label{power of log indep}
Let $m\in \Z_{\ge2}$, $K$ be an algebraic number field and $\alpha\in K\setminus\{0,-1\}$.  
We define the real numbers
\begin{align*}
&T(\alpha)={\rm{exp}}\left({\dfrac{2|{\rm{log}}(1+\alpha)|}{1+\sqrt{1+4|{\rm{log}}(1+\alpha)|}}}\right)(m-1)!,\\
&T^{(k)}(\alpha)=\dfrac{2^m(1+|\alpha^{(k)}|)(m-1)!}{|\alpha^{(k)}|} \ \text{for} \ 1\le k \le [K:\Q], \\
&\mathcal{A}^{(k)}(\alpha)=m(1+{\rm{log}}(2))+{\rm{log}}({\rm{den}}(\alpha))+{\rm{log}}(1+|\alpha^{(k)}|) \ \text{for} \ 1\le k \le [K:\Q],\\
&A(\alpha)=m{\rm{log}}\left(\dfrac{m}{|{\rm{log}}(1+\alpha)|}\right)-\left(\dfrac{m(1+\sqrt{1+4|{\rm{log}}(1+\alpha)|})}{2}+\dfrac{2|{\rm{log}}(1+\alpha)|}{1+\sqrt{1+4|{\rm{log}}(1+\alpha)|}} \right)-{\rm{log}}({\rm{den}}(\alpha)), \\
&\nu(\alpha):=A(\alpha)+\mathcal{A}^{(1)}(\alpha),\\
&\delta(\alpha):=A(\alpha)+\mathcal{A}^{(1)}(\alpha)-\dfrac{(m-1)\sum_{k=1}^{[K:\Q]}\mathcal{A}^{(k)}(\alpha) }{[K_{\infty}:\R]},
\end{align*}
where $K_{\infty}$ is the completion of $K$ with respect to the fixed embedding $\sigma:K\hookrightarrow \C$.
We assume $\dfrac{m}{|{\rm{log}}(1+\alpha)|}\ge 4$ and $\delta(\alpha)>0$, 
then the numbers $1, {\rm{log}}(1+\alpha), \ldots, {\rm{log}}^{m-1}(1+\alpha)$ are linearly independent over $K$. 
For any $\epsilon>0$, we take a natural number $n$ satisfying
\begin{align*}
&m\left[\sqrt{{\rm{log}}n}\cdot {\rm{exp}}\left(\dfrac{(\sqrt{322}-\sqrt{546})\sqrt{{\rm{log}}n}}{\sqrt{515}}\right)\right]<\dfrac{[K_{\infty}:\R]\delta^2(\alpha)\epsilon}{2(m-1)[K:\Q](2\nu(\alpha)+\epsilon\delta(\alpha))},\\
&{\rm{log}}\left(T(\alpha)n^{\tfrac{m}{2}} \left[\prod_{k=1}^{[K:\Q]}m!(T^{(k)}(\alpha))^{m-1}n^{2m(m-1)}\right]^{\tfrac{1}{[K_{\infty}:\R]}}\right)\le \dfrac{\epsilon \delta^2(\alpha)n}{4(2\nu(\alpha)+\epsilon \delta(\alpha))}.
\end{align*}
Then $H_0=\left(\dfrac{1}{2}{\rm{exp}}[{\delta(\alpha) n}]\right)^{\tfrac{[K_{\infty}:\R]}{[K:\Q]}}$ satisfies the following property$:$

For any $\boldsymbol{\beta}:=(\beta_0,\ldots,\beta_{m-1}) \in \mathcal{O}^{m}_K \setminus \{ \bold{0} \}$ satisfying $H_0<\mathrm{H}(\boldsymbol{\beta})\le H$, then we have
\begin{align*}
{\rm{log}}\left(\left|\sum_{i=0}^{m-1}\beta_i{\rm{log}}^i(1+\alpha)\right|\right)
>-\left(\dfrac{[K:\Q]\nu(\alpha)}{[K_{\infty}:\R]\delta(\alpha)}+\dfrac{\epsilon[K:\Q]}{2[K_{\infty}:\R]}\right){\rm{log}}(H).
\end{align*}
\end{theorem} 
We will prove Theorem $\ref{p power of log indep}$ in Section $6.1$. 
In the case of $\alpha$ is a rational number, we obtain the following corollary.
\begin{corollary} \label{corollary main theorem}
Let $m\in \mathbb{Z}_{\ge 2}$, $\epsilon>0$ and $\alpha=c/d\in {\mathbb{Q}}\setminus\{0,-1\}$ with $(c,d)=1$ and $d>0$.
Put
\begin{align*}
\nu(\alpha):=&m{\rm{log}}\left(\dfrac{m}{|{\rm{log}}(1+c/d)|}\right)-
\left(\dfrac{m(1+\sqrt{1+4|{\rm{log}}(1+c/d)|})^2+4|{\rm{log}}(1+c/d)|}{2(1+\sqrt{1+4|{\rm{log}}(1+c/d)|)}}\right)\\
&-{\rm{log}}(d)+m(1+{\rm{log}}(2))+{\rm{log}}(d)+{\rm{log}}(1+\left|{c}/{d}\right|).\\
\delta(\alpha):=&{m{\rm{log}}\left(\dfrac{m}{|{\rm{log}}(1+c/d)|}\right)}-
\left(\dfrac{m(1+\sqrt{1+4|{\rm{log}}(1+c/d)|})^2+4|{\rm{log}}(1+c/d)|}{2(1+\sqrt{1+4|{\rm{log}}(1+c/d)|)}}\right)\\
                     &{-{\rm{log}}(d)}-{(m-2)}\left(m(1+{\rm{log}}(2))+{{\rm{log}}(d)}+{\rm{log}}(1+\left|{c}/{d}\right|)\right).
\end{align*}
Suppose $\delta(\alpha)>0$. Then we have

$({\rm{i}})$  The complex numbers $1,{\rm{log}}(1+\alpha),\ldots,{\rm{log}}^{m-1}(1+\alpha)$ are linearly independent over $\Q$.

$({\rm{ii}})$ Let $n=n(\epsilon)$ be a natural number satisfying
\begin{align*}
&m\left[\sqrt{{\rm{log}}n}\cdot {\rm{exp}}\left(\dfrac{(\sqrt{322}-\sqrt{546})\sqrt{{\rm{log}}n}}{\sqrt{515}}\right)\right]<\dfrac{\delta^2(\alpha)\epsilon}{2(m-1)(2\nu(\alpha)+\epsilon\delta(\alpha))},\\
&{\rm{log}}\left(2^{(2m+1)(m-1)}(m!)^{m+1}(m-1)!^{m} \left(\dfrac{(d+|c|)}{|c|}\right)^{2(m-1)}\right)+
\left(\dfrac{2|{\rm{log}}(1+\tfrac{c}{d})|}{1+\sqrt{1+4|{\rm{log}}(1+\tfrac{c}{d})|}}\right)\\
&+\left(\dfrac{m}{2}+2m(m-1)\right){\rm{log}}(n)<\dfrac{\epsilon\delta(\alpha)^2n}{4(2\nu(\alpha)+\epsilon\delta(\alpha))}.
\end{align*}
Then for 
$H_0:=\tfrac{1}{2}{\rm{exp}}\left(\delta(\alpha)n\right)$ and 
$\bold{b}=(b_0,b_1,\ldots,b_{m-1}) \in \mathbb{Z}^{m}\setminus\{\bold{0}\}$ of $H_0<\mathrm{H}(\bold{b})\leq H$, we have
$${\rm{log}}\left(\left|\sum_{i=0}^{m-1}b_i{\rm{log}}^i(1+\alpha)\right|\right)>-\left( \dfrac{\nu(\alpha)}{\delta(\alpha)}+\dfrac{\epsilon}{2}\right) \cdot {\rm{log}} H.$$
\end{corollary}
\begin{remark}
We show that Corollary $\ref{corollary main theorem}$ gives an improvement of Theorem $\ref{N-W}$ for $1+\alpha\in \Q\setminus\{1,0\}$ which are sufficiently close to $1$ and $m\ge 3$. 
We compare the result of Theorem $\ref{N-W}$ with that of Corollary $\ref{corollary main theorem}$.
Let $m\in \mathbb{Z}_{\ge 2}$, $\epsilon>0$ and $\alpha=c/d\in {\mathbb{Q}}\setminus\{0,-1\}$ with $(c,d)=1$ and $d>0$.
For $\bold{b}=(b_0,b_1,\ldots,b_{m-1}) \in \mathbb{Z}^{m}\setminus\{\bold{0}\}$, by Theorem $\ref{N-W}$,
we obtain
\begin{align*}
{\rm{log}}\left(\left|\sum_{i=0}^{m-1}b_i{\rm{log}}^i(1+\alpha)\right|\right)\ge -C(\alpha)(m-1)^2(1+{\rm{log}}(m-1))^{-1}((m-1){\rm{log}}(m-1)+{\rm{log}}(L)),
\end{align*}
where $C(\alpha)$ is a positive number depending on $\alpha$ with $C(\alpha)>105500\cdot e^{{\rm{H}}(\alpha)}$ for $\mathrm{H}(\bold{b})\le H$.
Since we have $$-C(\alpha)(m-1)^2{\rm{log}}(H)\ge -C(\alpha)(m-1)^2(1+{\rm{log}}(m-1))^{-1}((m-1){\rm{log}}(m-1)+{\rm{log}}(L)),$$ 
we compare $C(\alpha)(m-1)^2$ and $\nu(\alpha)/\delta(\alpha)$.
Since we have 
\begin{align*}
\dfrac{\nu(\alpha)}{\delta(\alpha)}&\approx \dfrac{m\left({\rm{log}}(m)-{\rm{log}}(|c|)+{\rm{log}}(d)+{\rm{log}}(2)\right)}{{\rm{log}}(d)+m({\rm{log}}(m)-{\rm{log}}(|c|)-1-(m-2){\rm{log}}(2))}
                                              \approx m,
\end{align*}
if $|\alpha|=|c|/d$ is sufficiently close to $0$, Corollary  $\ref{corollary main theorem}$ improves Theorem $\ref{N-W}$ for $1+\alpha\in \Q\setminus\{1,0\}$ which are sufficiently close to $1$ and $m\ge 3$. 
\end{remark}
\bigskip
Second, we introduce a $p$-adic version of Theorem $\ref{power of log indep}$.
\begin{theorem}\label{p power of log indep}
Let $m\in\Z_{\ge2}$, $p$ be a prime number, $K$ an algebraic number field and $\alpha\in K\setminus\{0,-1\}$ with $|\alpha|_p<1$.  
We use the same notations as in Theorem $\ref{power of log indep}$.
We also define the real numbers
\begin{align*}
&T_p(\alpha)=\dfrac{(2m)^{m-1}}{|\alpha|_p},\\
&A_p(\alpha)=-m{\rm{log}}(|\alpha|_p),\\
&\nu_p(\alpha)=A_p(\alpha),\\
&\delta_p(\alpha)=A_p(\alpha)-\dfrac{(m-1)\sum_{k=1}^{[K:\Q]}\mathcal{A}^{(k)}(\alpha)}{[K_{\infty}:\R]}.
\end{align*}
We assume $\delta_p(\alpha)>0$, 
then the numbers $1, {\rm{log}}_p(1+\alpha), \ldots, {\rm{log}}^{m-1}_p(1+\alpha)$ are linearly independent over $K$. 
For any $\epsilon>0$, we take a natural number $n$ satisfying
\begin{align*}
&\dfrac{1}{{\rm{log}}|\alpha|^{-1}_p}+\dfrac{1}{m}\le n,\\
&mn\left[\sqrt{{\rm{log}}n}\cdot {\rm{exp}}\left(\dfrac{(\sqrt{322}-\sqrt{546})\sqrt{{\rm{log}}n}}{\sqrt{515}}\right)\right] \le \dfrac{\epsilon \delta^2_p(\alpha)[K_{p}:\Q_p]n}{2(m-1)(2\nu_p(\alpha)+\epsilon \delta_p(\alpha))[K:\Q]},\\
&{\rm{log}}\left(T_{p}(\alpha)n^{m-1} \left[\prod_{k=1}^{[K:\Q]}m!(T^{(k)}(\alpha))^{m-1} n^{2m(m-1)}\right]^{\tfrac{1}{[K_{p}:\Q_p]}}\right) \le \dfrac{\epsilon \delta^2_p(\alpha)n}{4 (2\nu_p(\alpha)+\epsilon \delta_p(\alpha)},
\end{align*}
where $K_{p}$ is the completion of $K$ with respect to the fixed embedding $\sigma_p:K\hookrightarrow \C_p$.
Then $H_0=\left(\dfrac{1}{2}{\rm{exp}}[{\delta_p(\alpha) n}]\right)^{\tfrac{[K_{p}:\Q_p]}{[K:\Q]}}$ satisfies the following property$:$

For any $\boldsymbol{\beta}:=(\beta_0,\ldots,\beta_{m-1}) \in \mathcal{O}^{m}_K \setminus \{ \bold{0} \}$ satisfying $H_0< \mathrm{H}(\boldsymbol{\beta})\le H$, then we have
\begin{align*}
{\rm{log}}\left(\left|\sum_{i=0}^{m-1}\beta_i{\rm{log}}^i_p(1+\alpha)\right|_p\right)>
-\left(\dfrac{[K:\Q]\nu_p(\alpha)}{[K_{p}:\Q_p]\delta_p(\alpha)}+\dfrac{\epsilon[K:\Q]}{2[K_p:\Q_p]}\right)
{\rm{log}}(H).
\end{align*}
\end{theorem}
We will prove Theorem $\ref{p power of log indep}$ in Section $6.2$. 
\section{Pad\'{e} approximations of formal power series}
In this section, we recall the definition and basic properties of Pad\'{e} approximation of formal power series.
In the following of this section, we use $K$ as a field with characteristic $0$.
\begin{lemma} \label{Pade}
Let $m\in\N$ and $\bold{f}=(f_1(z),\ldots,f_m(z))\in K[[z]]^m$. For $\bold{n}:=(n_1,\ldots,n_m)\in \Z^m_{\ge0}$, there exists a family of polynomials $(A_1(z),\ldots,A_m(z))\in K[z]^m$ satisfying the following properties$:$
\begin{align*}
&({\rm{i}}) \ (A_1(z),\ldots,A_m(z))\neq(0,\ldots,0),\\
&({\rm{ii}}) \ {\rm{deg}}A_j(z)\le n_j \ \text{for} \ 1\le j\le m,\\
&({\rm{iii}}) \ {\rm{ord}}\sum_{j=1}^{m}A_j(z)f_j(z)\ge \sum_{j=1}^m(n_j+1)-1.
\end{align*}
\end{lemma}
In this article, we call the polynomials $(A_1(z),\ldots,A_m(z))\in K[z]^m$ satisfying the conditions $({\rm{i}}),({\rm{ii}}),({\rm{iii}})$ in Lemma $\ref{Pade}$ as a weight $\bold{n}$ Pad\'{e} approximants  of $\bold{f}$. For a weight $\bold{n}$ Pad\'{e} approximants of $\bold{f}$, $(A_1(z),\ldots,A_m(z))$, we call the formal power series $\sum_{j=1}^{m}A_j(z)f_j(z)$ as a weight $\bold{n}$ Pad\'{e} approximation of $\bold{f}$. 
\begin{definition} 
Let $m\in\Z_{\ge1}$ and $\bold{f}:=(f_1(z),\ldots,f_m(z))\in K[[z]]^m$.

$({\rm{i}})$ Let $\bold{n}=(n_1,\ldots,n_m)\in \Z^m_{\ge0}$. We say $\bold{n}$ is normal with respect to $\bold{f}$ if for any weight $\bold{n}$ Pad\'{e} approximation $R(z)$ of $\bold{f}$ satisfy the equality $${\rm{ord}}R(z)=\sum_{j=1}^m(n_j+1)-1.$$

$({\rm{ii}})$  We call $\bold{f}$ is perfect if any indices $\bold{n}\in \Z^m_{\ge0}$ are normal with respect to $\bold{f}$.
\end{definition}
\begin{remark} \label{remark bij} 
Let $m\in\N$, $\bold{n}=(n_1,\ldots,n_m)\in \Z^m_{\ge0}$ and $\bold{f}=(f_j(z):=\sum_{k=0}^{\infty}f_{j,k}z^k)_{1\le j \le m} \in K[[z]]^m$. 
We put $N=\sum_{j=1}^m(n_j+1)$. For $r\in \Z_{\ge0}$, we define a $N\times (r+1)$ matrix $A_{\bold{n},r}(\bold{f})$ by
\begin{equation*}
                     A_{\bold{n},r}(\bold{f}):={\begin{pmatrix}
                     f_{1,0}& 0 & \dots & 0 & \ldots & f_{m,0}& 0 & \dots & 0\\
                     f_{1,1}& f_{1,0} & \dots & 0 &  \ldots & f_{m,1}& a_{m,0} & \dots & 0\\
                     \vdots & \vdots & \ddots & \vdots  & \ddots  & \vdots & \vdots & \ddots &\vdots \\
                     f_{1,r}& f_{1,r-1} & \dots & f_{1,r-n_1} & \ldots & f_{m,r} & f_{m,r-1} & \dots & f_{m,r-n_m}\\
                     \end{pmatrix}},
                     \end{equation*}
where $a_{j,k}=0$ if $k<0$ for $1\le j \le m$. Then we have the following bijection:
\begin{align*}
&\phi^{(\bold{n})}_{\bold{f}}: {\rm{ker}}(A_{\bold{n},N-2}(\bold{f}))\setminus \{\bold{0}\}\longrightarrow \left\{(A_j(z))\in K[z]^m \middle| \sum_{j} A_j(z)f_j(z) \ \text{is a weight} \ \bold{n}\text{ Pad\'{e} approximation of} \ \bold{f}\right\}\\ 
&{}^{t}(a_{1,0},\ldots,a_{1,n_1},\ldots,a_{m,0}, \ldots, a_{m,n_m})\mapsto \left(A_j(z):=\sum_{k=0}^{n_j}a_{j,k}z^k\right)_{1\le j \le m}. \nonumber
\end{align*}
Note that the indice $\bold{n}$ is normal with respect to $\bold{f}$ is equivalent to $A_{\bold{n},N-1}(\bold{f})\in {\rm{GL}}_N(K)$.
\end{remark}
\begin{lemma} \label{cor fund Pade}
Let $m\in \N$, $\bold{f}:=(f_1,\ldots,f_m)\in K[[z]]^m$ and $\bold{n}:=(n_1,\ldots,n_m)\in \N^m$. 
Put $\bold{n}_i:=(n_1,\ldots,n_{i-1},n_i+1,n_{i+1},\ldots,n_m)\in \N^{m}$ for $1 \le i \le m$.
Suppose $\bold{n}$ is normal with respect to $\bold{f}$. Then we have $${\rm{deg}}A_i(z)=n_i+1,$$
for any weight $\bold{n}_i$ Pad\'{e} approximants $(A_1(z),\ldots,A_m(z))$ of $\bold{f}$. 
\end{lemma}
\begin{proof}
Put $N:=\sum_{j=1}^m (n_j+1)$. Since $\bold{n}$ is normal with respect to $\bold{f}$, we have  
${\rm{dim}}_K{\rm{ker}}(A_{\bold{n},N-1}(\bold{f}))=0$.
Suppose there exist $1\le i \le m$ and a weight $\bold{n}_i$ Pad\'{e} approximants $(A_1(z),\ldots,A_m(z))$ of $\bold{f}$ satisfying ${\rm{deg}}A_i<n_i+1$.
Put $${}^{t}(a_{1,0},\ldots,a_{1,n_1},\ldots, a_{i,0},\ldots,a_{i,n_i},0,\ldots,a_{m,0},\ldots,a_{m,n_m}):=(\phi^{(\bold{n}_i)}_{\bold{f}})^{-1}(A_1(z),\ldots,A_m(z)).$$
Then we have $${}^{t}(a_{1,0},\ldots,a_{1,n_1},\ldots, a_{i,0},\ldots,a_{i,n_i},\ldots,a_{m,0},\ldots,a_{m,n_m})\in {\rm{ker}}(A_{\bold{n},N-1}(\bold{f}))\setminus\{\bold{0}\}.$$ This is a contradiction. This completes the proof of Lemma $\ref{cor fund Pade}$.
\end{proof}
\section{Pad\'{e} approximation of exponential functions}
In this section, we recall some properties of Pad\'{e} approximation of exponential functions. 
We quote some propositions for the Pad\'{e} approximation of exponential functions in \cite{J}.
\begin{proposition} \label{perfect e} $($cf.  \rm{\cite[Theorem $1.2.1$]{J}}$)$
Let $n$ be a natural number and $\omega_1,\ldots,\omega_n$ pairwise distinct complex numbers. Then the functions $e^{\omega_1z},\ldots,e^{\omega_n z}$ are perfect. Especially, for $l\in \N$, the functions $1,e^z,\ldots,e^{lz}$ are perfect.
\end{proposition}
Let $\omega_1,\ldots,\omega_n$ be pairwise distinct complex numbers. 
Explicit construction of Pad\'{e} approximations of $e^{\omega_1 z},\ldots,e^{\omega_n z}$ are given by Hermite as follows.
\begin{proposition} \label{Pade e} $(cf. $\rm{\cite[p. $242$]{J}}$)$
Let $\bold{m}:=(m_1,\ldots,m_n)\in \Z^{n}_{\ge0}$ and $\{a_{h,j}(\bold{m}, \boldsymbol{\omega})\}_{1\le h \le n, 1\le j \le m_h+1}$ be the family of complex numbers satisfying the following equality$:$
\begin{align} \label{coefficients}
\dfrac{1}{\prod_{h=1}^n(x-\omega_h)^{m_h+1}}=\sum_{h=1}^n \sum_{j=1}^{m_h+1}\dfrac{a_{h,j}(\bold{m},\boldsymbol{\omega})}{(x-\omega_h)^j}.
\end{align}
Then the formal power series $$S(z):=\sum_{h=1}^n\left(\sum_{j=0}^{m_h}a_{h,j+1}(\bold{m},\boldsymbol{\omega})\dfrac{z^j}{j!}\right)e^{\omega_hz},$$
is a weight $\bold{m}$ Pad\'{e} approximation of $e^{\omega_1z},\ldots,e^{\omega_nz}$.
\end{proposition}
\section{Pad\'{e} approximations of power of logarithm functions}
In this section, we construct Pad\'{e} approximations of $1,{\rm{log}}(1+z),\ldots, {\rm{log}}^{m-1}(1+z)$ for $m\in \Z_{\ge2}$ by using that of exponential functions obtained in Proposition $\ref{Pade e}$. 
\begin{lemma} \label{normality}
Let $f(z)\in K[[z]]$. Suppose there exists $g(z)\in K[[z]]$ satisfying $f(g(z))=g(f(z))=z$.
Put $\tilde{g}(z):=g(z)+1$ and assume $1,\tilde{g}(z),\ldots,\tilde{g}^l(z)$ are perfect for any $l\in \N$. 
Then the any indices $\bold{n}\in \{(n_0,n_1,\ldots,n_{m-1})\in \Z^m_{\ge0}| \ n_0\ge n_1\ge \ldots \ge n_{m-1}\}$ are normal with respect to $(1,f(z),\ldots,f^{m-1}(z))$ for any $m\in \Z_{\ge 2}$.
\end{lemma} 
\begin{proof}
Denote the set $\{(n_0,n_1,\ldots,n_{m-1})\in \Z^m_{\ge0}| \ n_0\ge n_1\ge \ldots \ge n_{m-1}\}$ by $\mathcal{X}_m$. 
Let $$\bold{n}:=((n_0)_{r_0},(n_1)_{r_1}\ldots, (n_s)_{r_s}) \in \mathcal{X}_m \ \text{for} \ n_0>n_1>\ldots >n_s,$$
where $(n_i)_{r_i}=(n_i,\ldots,n_i)\in \Z^{r_i}_{\ge 0}$ for $0\le i \le s$. We put  
\begin{align*}
&\bold{m}:=((m-1)_{n_s+1},(\sum_{i=0}^{s-1}r_i-1)_{n_{s-1}-n_s},(\sum_{i=0}^{s-2}r_i-1)_{n_{s-2}-n_{s-1}},\ldots,(r_0-1)_{n_0-n_1})\in \mathcal{X}_{n_0+1},\\
&\bold{f}:=(1,f(z),\ldots,f^{m-1}(z)),\\
&\tilde{\bold{g}}:=(1,\tilde{g}(z),\ldots,\tilde{g}^{n_0}(z)),\\
&V_{\bold{n}}:=\left\{R(z)=\sum_{j=0}^{m-1}A_j(z)f^j(z) \middle| R(z) \ \text{is a weight} \ \bold{n} \ \text{Pad\'{e} approximation of} \ \bold{f}\right\},\\
&W_{\bold{m}}:=\left\{\mathcal{R}(z)=\sum_{j=0}^{n_0}\mathcal{A}_j(z)\tilde{g}^j(z) \middle| \mathcal{R}(z) \ \text{is a weight} \ \bold{m} \ \text{Pad\'{e} approximation of} \ \tilde{\bold{g}}\right\}.
\end{align*}
We define the $K$-isomorphism $\Psi$ by $$\Psi:K[[z]]\longrightarrow K[[z]], \ \sum_{k=0}^{\infty}a_k z^k\mapsto \sum_{k=0}^{\infty}a_k \tilde{g}^k(z).$$
Note that $\Psi$ is an order preserving map, namely we have ${\rm{ord}}F(z)={\rm{ord}}\Psi(F(z))$ for $F(z)\in K[[z]]$.
We prove that $\Psi$ induces the bijection 
$\Psi:V_{\bold{n}}\longrightarrow W_{\bold{m}}$. Let $R(z)=\sum_{j=0}^{m-1}A_j(z)f^j(z)\in V_{\bold{n}}$ and put 
$A_j(z)=\sum_{h=0}^{n_{0}}a_{h,j}(1+z)^h$ with for $0\le j \le m-1$.
Then we obtain 
\begin{align} \label{R 1}
&R(z)=\\
&\sum_{j=0}^{r_0-1}(\sum_{h=0}^{n_0}a_{h,j}(1+z)^h)f^{j}(z)+\sum_{j=r_0}^{r_0+r_1-1}(\sum_{h=0}^{n_1}a_{h,j}(1+z)^h)f^{j}(z)+\ldots+\sum_{j=r_0+\ldots+r_{s-1}}^{m-1}(\sum_{h=0}^{n_s}a_{h,j}(1+z)^h)f^{j}(z).\nonumber
\end{align} 
Using $(\ref{R 1})$, we have
\begin{align}\label{Psi R}
&\Psi(R)(z)=\nonumber\\
&\sum_{j=0}^{r_0-1}(\sum_{h=0}^{n_0}a_{h,j}\tilde{g}^h(z))z^j+\sum_{j=r_0}^{r_0+r_1-1}(\sum_{h=0}^{n_0}a_{h,j}\tilde{g}^h(z))z^j+\ldots+\sum_{j=r_0+\ldots+r_{s-1}}^{m-1}(\sum_{h=0}^{n_s}a_{h,j}\tilde{g}^h(z))z^j=\nonumber\\
&\sum_{h=0}^{n_s}\left(\sum_{j=0}^{m-1}a_{h,j}z^j\right)\tilde{g}^h(z)+\sum_{h=n_s+1}^{n_{s-1}}\left(\sum_{j=0}^{r_0+\ldots+r_{s-1}-1}a_{h,j}z^j\right)\tilde{g}^h(z)+\ldots+\sum_{h=n_1+1}^{n_0}\left(\sum_{j=0}^{r_0-1}a_{h,j}z^j\right)\tilde{g}^h(z).
\end{align}
Since $\Psi$ is order preserving map, the equality $(\ref{Psi R})$ shows that $\Psi(R)$ is a weight $\bold{m}$ Pad\'{e} approximation of $\tilde{\bold{g}}$. 
Then we have $\Psi(V_{\bold{n}})\subseteq W_{\bold{m}}$. 
By the similar way, we also obtain $W_{\bold{m}}\subseteq \Psi(V_{\bold{n}})$. Then the map $\Psi:V_{\bold{n}}\longrightarrow W_{\bold{m}}$ is bijection.
Since $\bold{m}$ is normal with respect to $\tilde{\bold{g}}$, we have ${\rm{ord}}S(z)=\sum_{i=0}^s (n_i+1)r_i-1$ for all $S(z)\in W_{\bold{m}}$. 
Since the bijection $\Psi:V_{\bold{n}}\longrightarrow W_{\bold{m}}$ is order preserving map, we also obtain ${\rm{ord}}R(z)=\sum_{i=0}^s (n_i+1)r_i-1$ for all $R(z)\in V_{\bold{n}}$. 
This shows that the indice $\bold{n}$ is normal with respect to $\bold{f}$. This completes the proof of Lemma $\ref{normality}$.
\end{proof} 
\begin{proposition} $($cf. {\rm{\cite[Theorem $1.2.3$]{J}}}$)$ \label{diagonal normality log}
Let $m\in \Z_{\ge 2}$. Denote the set $\{(n_0,\ldots,n_{m-1})\in \Z^m_{\ge0}| \ n_0\ge n_1\ge \ldots \ge n_{m-1}\}$ by $\mathcal{X}_m$. 
Then any indices $\bold{n}\in \mathcal{X}_m$ are normal with respect to $(1,{\rm{log}}(1+z),\ldots,{\rm{log}}^{m-1}(1+z))$. 
\end{proposition}
\begin{proof}
By Proposition $\ref{perfect e}$, we have  $1,e^z,\ldots,e^{lz}$ are perfect for $l\in \N$.
Using Lemma $\ref{normality}$ for $f(z):={\rm{log}}(1+z)$ and $g(z):=e^z-1$, any indices $\bold{n}\in \mathcal{X}_m$ are normal with respect to $(1,{\rm{log}}(1+z),\ldots,{\rm{log}}^{m-1}(1+z))$.  This completes the proof of Proposition $\ref{diagonal normality log}$.
\end{proof}
Let $m\in \Z_{\ge2}$ and $\bold{n}\in \mathcal{X}_m$. We obtain a weight $\bold{n}$ Pad\'{e} approximation of $1,{\rm{log}}(1+z),\ldots,{\rm{log}}^{m-1}(1+z)$ as follows:
\begin{proposition} \label{pade log}
Let $\{r_i\}_{0\le i \le s}\subset \N$ and $\{n_i\}_{0\le i \le s}\subset \Z_{\ge0}$ satisfying $r_0+\ldots+r_s=m$ and $n_0>\ldots>n_s$. 
Put 
\begin{align*}
&\bold{n}:=((n_0)_{r_0},\ldots,(n_s)_{r_s})\in \mathcal{X}_m\\
&\bold{m}:=((m-1)_{n_s+1},(\sum_{i=0}^{s-1}r_i-1)_{n_{s-1}-n_s},(\sum_{i=0}^{s-2}r_i-1)_{n_{s-2}-n_{s-1}},\ldots,(r_0-1)_{n_0-n_1}),\\
&\boldsymbol{\omega}:=(0,1,\ldots,n_0).
\end{align*}
We define the family of rational numbers $\{a_{h,j}(\bold{m}, \boldsymbol{\omega})\}_{0\le h \le n_0+1, 1\le j \le m-1}$ as follows$:$
$$\dfrac{1}{\prod_{h=0}^{n_s}(x-h)^{m}}\times \dfrac{1}{\prod_{h=n_s+1}^{n_{s-1}}(x-h)^{\sum_{i=1}^{s-1}r_i}} \times \ldots \times \dfrac{1}{\prod_{h=n_1+1}^{n_{0}}(x-h)^{r_0}}
=\sum_{h=0}^{n_0} \sum_{j=1}^{m-1}\dfrac{a_{h,j}(\bold{m},\boldsymbol{\omega})}{(x-h)^j}.$$
Then the formal power series 
\begin{align}\label{Pade log power}
R(z):=\sum_{j=0}^{m-1} \left( \dfrac{\sum_{h=0}^{n_0}a_{h,j+1}(\bold{m},\boldsymbol{\omega})(1+z)^h}{j!}\right) {\rm{log}}^{j}(1+z),
\end{align}
is a weight $\bold{n}$ Pad\'{e} approximation of $(1,{\rm{log}}(1+z),\ldots,{\rm{log}}^{m-1}(1+z))$.
\end{proposition}
\begin{proof}
We define a $\overline{\Q}$-isomorphism $\Psi$ by $$\Psi:\overline{\Q}[[z]]\longrightarrow \overline{\Q}[[z]], \ z\mapsto e^{z}-1.$$
By Proposition $\ref{Pade e}$, the formal power series 
\begin{align} \label{explicit Pade e}
\mathcal{R}(z):=\sum_{h=0}^{n_0}\left(\sum_{j=0}^{m-1}a_{h,j+1}(\bold{m},\boldsymbol{\omega})\dfrac{z^j}{j!}\right)e^{hz},
\end{align} is a weight $\bold{m}$ Pad\'{e} approximation of $1,e^z,\ldots,e^{n_0z}$. By the proof of Lemma $\ref{normality}$, we have 
\begin{align} \label{R}
\Psi^{-1}(\mathcal{R}(z))=\sum_{j=0}^{m-1} \left( \dfrac{\sum_{h=0}^{n_0}a_{h,j+1}(\bold{m},\boldsymbol{\omega})(1+z)^h}{j!}\right) {\rm{log}}^{j}(1+z),
\end{align}
is a weight $\bold{n}$ Pad\'{e} approximation of $(1,{\rm{log}}(1+z),\ldots,{\rm{log}}^{m-1}(1+z))$. Since the right hand side of the equality $(\ref{R})$ is the formal power series $R(z)$ defined in $(\ref{Pade log power})$, this completes the proof of Proposition $\ref{pade log}$. 
\end{proof}
\begin{remark}
Let $(n_0,\ldots,n_m)\in \mathcal{X}_m$ and $R(z)$ and $\mathcal{R}(z)$ be the formal power series defined in $(\ref{Pade log power})$ and $(\ref{explicit Pade e})$ respectively.   
Put $N:=\sum_{j=0}^{m-1}(n_j+1)$. We have $\mathcal{R}(z)=\dfrac{z^{N-1}}{(N-1)!}+\text{(higher order term)}$ (see p. $242$ \cite{J}). Then by the definition of $R(z)$, we have 
\begin{align} \label{first term R}
R(z)=\dfrac{z^{N-1}}{(N-1)!}+\text{(higher order term)}.
\end{align}
On the other hand, in p.~$245$ \cite{J}, Jager proved that the function $$r(z):=\dfrac{1}{2\pi\sqrt{-1}}\int_{C}\dfrac{(1+z)^x}{\prod_{j=0}^{m-1}\prod_{h=0}^{n_j}(x-h)}dx,$$
where $C$ is a contour with positive orientation enclosing the set $\{0,1,\ldots,n_0\}$, is a weight $\bold{n}$ Pad\'{e} approximation of $(1,{\rm{log}}(1+z),\ldots,{\rm{log}}^{m-1}(1+z))$ and $r(z)$ satisfies 
\begin{align} \label{first term r}
r(z)=\dfrac{z^{N-1}}{(N-1)!}+\text{(higher order term)} \ \text{for} \ z\in\{z\in\C \mid |z|<1\}.
\end{align}
Since $\bold{n}$ is normal with respect to $(1,{\rm{log}}(1+z),\ldots,{\rm{log}}^{m-1}(1+z))$, a weight $\bold{n}$ Pad\'{e} approximation of  $(1,{\rm{log}}(1+z),\ldots,{\rm{log}}^{m-1}(1+z))$ is uniquely determined up to constant. Thus, by $(\ref{first term R})$ and $(\ref{first term r})$, we obtain
\begin{align} \label{integral rep R}
R(z)=\dfrac{1}{2\pi\sqrt{-1}}\int_{C}\dfrac{(1+z)^x}{\prod_{j=1}^{m-1}\prod_{h=0}^{n_j}(x-h)}dx.
\end{align}
\end{remark}
\section{Estimations}
From this section to the last section, we use the following notations for $m\in \Z_{\ge2}$, $n\in \Z_{\ge0}$ and $1\le i \le m$:
\begin{align*}
&d_{n+1}:={\rm{l.c.m.}}(1,2,\ldots,n+1),\\
&\bold{n}_i:=(\overbrace{n+1,\ldots,n+1}^{i},n,\ldots,n)\in \Z^{m}_{\ge0},\\
&\bold{m}_i:=(m-1,\ldots,m-1,i-1)\in \Z^{n+2}_{\ge0},\\
&\boldsymbol{\omega}:=(0,\ldots,n,n+1),\\
&Q_{m,i,n+1}(x):=\left[\prod_{h=0}^n(x-h)^m\right]\times (x-n-1)^i.
\end{align*}
We define the set of rational numbers $\{a_{h,j}(\bold{m}_i,\boldsymbol{\omega})\}_{1\le i \le m, 0\le h \le n+1,1\le j \le m}$ satisfying the equality
\begin{align*}
\dfrac{1}{Q_{m,i,n+1}(x)}=\sum_{h=0}^{n+1}\sum_{j=1}^{m}\dfrac{a_{h,j}(\bold{m}_i,\boldsymbol{\omega})}{(x-h)^j} \ \text{for} \ 1\le i \le m.
\end{align*}
By Proposition $\ref{Pade e}$ and Proposition $\ref{pade log}$, the formal power series
\begin{align}
&\mathcal{R}_{i,n+1}(z):=\sum_{h=0}^{n+1}\left(\sum_{j=0}^{m-1}a_{h,j+1}(\bold{m}_i,\boldsymbol{\omega})\dfrac{z^j}{j!}\right)e^{hz}, \label{exp pade S}\\
&R_{i,n+1}(z):=\sum_{j=0}^{m-1}\left(\dfrac{\sum_{h=0}^{n+1}a_{h,j+1}(\bold{m}_i,\boldsymbol{\omega})(1+z)^h}{j!}\right){\rm{log}}^{j}(1+z), \label{log pade R}
\end{align}
are weight $\bold{m}_i$ Pad\'{e} approximation of $1,e^z,\ldots,e^{(n+1)z}$ and weight $\bold{n}_i$ Pad\'{e} approximation of $1,{\rm{log}}(1+z),\ldots,{\rm{log}}^{m-1}(1+z)$ respectively. 
We define
\begin{align} \label{coeff polynomial}
A_{i,j,n+1}(z):=\dfrac{\sum_{h=0}^{n+1}a_{h,j+1}(\bold{m}_i,\boldsymbol{\omega})(1+z)^h}{j!} \ \text{for} \ 1\le i \le m, 0\le j \le m-1.
\end{align}
\begin{lemma} \label{denominator} $($cf. {\rm{\cite[Theorem $1 (a)$]{M2}}}$)$
We use the notations as above. For any $1\le i \le m$, $0\le j \le m-1$ and $0\le h \le n+1$, we have 
\begin{align} \label{tisai denominator}
d^m_{n+1}(n+1)!^m a_{h,j}(\bold{m}_i,\boldsymbol{\omega})\in \Z.
\end{align}
Especially, for an algebraic number field $K$ and an element $\alpha\in K$, we have 
\begin{align} \label{denomi2}
d^m_{n+1}(n+1)!^m(m-1)!{\rm{den}}^{n+1}(\alpha)A_{i,j,n+1}(\alpha)\in \mathcal{O}_K,
\end{align}
for $1\le i \le m$, $0\le j \le m-1$.
\end{lemma}
\begin{proof}
Recall that,  by the definition of $a_{h,j}(\bold{m}_i,\boldsymbol{\omega})$, we have 
\begin{align} \label{fukusyu}
\dfrac{1}{Q_{m,i,n+1}(x)}=\sum_{j=1}^{m}\sum_{h=0}^{n+1} \dfrac{a_{h,j}(\bold{m}_i,\boldsymbol{\omega})}{(x-h)^j}.
\end{align}
Fix a nonzero integer $\lambda$ satisfying $0\le \lambda \le n$. By the definition of $Q_{m,i,n+1}(x)$, we have the following equalities:
\begin{align}
&\dfrac{1}{Q_{m,i,n+1}(x)}\nonumber \\
&=\dfrac{1}{(x-\lambda)^m}\prod_{\delta=1}^{\lambda}\dfrac{1}{(x-\lambda+\delta)^m}\prod_{\nu=1}^{n-\lambda}\dfrac{1}{(x-\lambda-\nu)^m} 
\dfrac{1}{(x-\lambda-(n+1-\lambda))^i} \nonumber\\
&=\dfrac{(-1)^{(n-\lambda)m+i}}{\lambda !^m(n-\lambda)!^m(n+1-\lambda)^i} \dfrac{1}{(x-\lambda)^m}\prod_{\delta=1}^{\lambda}\left(1+\dfrac{x-\lambda}{\delta}\right)^{-m} \prod_{\nu=1}^{n-\lambda}\left(1-\dfrac{x-\lambda}{\nu}\right)^{-m} \left(1-\dfrac{x-\lambda}{n+1-\lambda}\right)^{-i}. \label{lambda exp}
\end{align}
Since we have
\begin{align*}
\dfrac{d_{n+1}}{\delta}, \dfrac{d_{n+1}}{\nu}\in \Z \ \text{for} \ \delta=1,\ldots,\lambda \ \text{and} \ \nu=1,\ldots,n-\lambda,n+1-\lambda,
\end{align*} 
then there exist a set of integers $\{c_{i,k}\}_{k\in\Z_{\ge0}}$ satisfying 
\begin{align} \label{bekikyuusuu}
\prod_{\delta=1}^{\lambda}\left(1+\dfrac{d_{n+1}}{\delta}t\right)^{-m} \prod_{\nu=1}^{n-\lambda}\left(1-\dfrac{d_{n+1}}{\nu}t\right)^{-m} \left(1-\dfrac{d_{n+1}}{n+1-\lambda}t\right)^{-i}=\sum_{k=0}^{\infty}c_{i,k}t^k,
\end{align} 
where $t$ is an intermediate. Substituting $t=\dfrac{x-\lambda}{d_{n+1}}$ in the equality $(\ref{bekikyuusuu})$, we obtain 
\begin{align}\label{bekiyuusuu2}
\prod_{\delta=1}^{\lambda}\left(1+\dfrac{x-\lambda}{\delta}\right)^{-m} \prod_{\nu=1}^{n-\lambda}\left(1-\dfrac{x-\lambda}{\nu}\right)^{-m} \left(1-\dfrac{x-\lambda}{n+1-\lambda}\right)^{-i}=\sum_{k=0}^{\infty}c_{i,k}d^{-k}_{n+1}(x-\lambda)^k.
\end{align}
Substituting $(\ref{bekiyuusuu2})$ for the equality $(\ref{lambda exp})$ and compare the equality $(\ref{fukusyu})$ and $(\ref{lambda exp})$, we have
\begin{align} \label{key kill denomi}
a_{\lambda,j}(\bold{m}_i,\boldsymbol{\omega})=\dfrac{(-1)^{(n-\lambda)m+i}}{\lambda !^m(n-\lambda)!^m(n+1-\lambda)^i}d^{-m+j}_{n+1} c_{i,m-j}.
\end{align}
By the relation
$(n+1)!^m\dfrac{1}{\lambda !^m(n-\lambda)!^m(n+1-\lambda)^i}\in \Z$ and the equality $(\ref{key kill denomi})$, we obtain 
$$d^m_{n+1}(n+1)!^m a_{\lambda,j}(\bold{m}_i,\boldsymbol{\omega})\in \Z \ \text{for} \ 0\le \lambda \le n.$$
In the case of $\lambda=n+1$, by using the same method as above, we also obtain $$d^m_{n+1}(n+1)!^m a_{n+1,j}(\bold{m}_i,\boldsymbol{\omega})\in \Z.$$  This completes the proof of $(\ref{tisai denominator})$. The latter assertions are obtained by $(\ref{tisai denominator})$ and the definition of $A_{i,j,n+1}(z)$. This completes the proof of Lemma $\ref{denominator}$.
\end{proof}
\begin{lemma} \label{keisu ookisa} $($cf. {\rm{\cite[Theorem $1 (b)$]{M2}}}$)$
Let $\alpha$ be a complex number. Then we have
\begin{align}
|A_{i,j,n+1}(\alpha)|\le \dfrac{2^m(1+|\alpha|)}{|\alpha|}(n+1)^{m} ((1+|\alpha|)2^{m})^{n+1}n!^{-m}
\end{align}
for any $1\le i \le m$, $0\le j \le m-1$ and $n\in \Z_{\ge 0}$.
\end{lemma}
\begin{proof}
In our proof of Lemma $\ref{keisu ookisa}$, we refer some of the arguments of {\cite[Theorem $1 (b)$]{M2}}.
First we remark that the $a_{\lambda,j}(\bold{m}_i,\boldsymbol{\omega})$ can be represented as follows:
\begin{align} \label{residue}
a_{\lambda,j}(\bold{m}_i,\boldsymbol{\omega})=\dfrac{1}{2\pi\sqrt{-1}}\int_{|z-\lambda|=\tfrac{1}{2}}(z-\lambda)^{j-1}\dfrac{1}{Q_{m,i,n+1}(z)}dz.
\end{align}
Let $\lambda$ be a natural number satisfying $0\le \lambda \le n+1$. Then by the equality $(\ref{residue})$, we have the following inequality:
\begin{align} \label{a1}
|a_{\lambda,j}(\bold{m}_i,\boldsymbol{\omega})|\le\dfrac{1}{2\pi}2^{1-j}\pi\cdot {\rm{sup}}_{|z-\lambda|=1/2}\left|\dfrac{1}{Q_{m,i,n+1}(z)}\right|.
\end{align}
Next, we estimate a lower bound of ${\rm{sup}}_{|z-\lambda|=1/2}\left|{Q_{m,i,n+1}(z)}\right|$. Since we have the inequality $$|z-h|=|z-\lambda+\lambda-h|\ge |\lambda-h|-\tfrac{1}{2},$$ for natural number $h$ satisfying $0\le h \le n+1$ and $h\neq \lambda$ and $z\in\{z\in \C| \ |z-\lambda|=\tfrac{1}{2}\}$, we obtain the following inequalities: 
\begin{align} \label{ineq Q}
|Q_{m,i,n+1}(z)|&\ge\begin{cases}
\left(\prod_{\delta=1}^{\lambda}(\delta-\tfrac{1}{2})\right)^m\left(\tfrac{1}{2}\right)^m\left(\prod_{\nu=1}^{n-\lambda}(\nu-\tfrac{1}{2})\right)^{m}(n+1-\lambda-\tfrac{1}{2})^i & \ \text{if} \ 0\le \lambda \le n, \\
\left(\prod_{\delta=1}^{n+1}(\delta-\tfrac{1}{2})\right)^m \left(\tfrac{1}{2}\right)^i  & \ \text{if} \ \lambda=n+1.
\end{cases}
\end{align}
In the case of $0\le \lambda \le n$, using the inequality $(\ref{ineq Q})$, we obtain  
\begin{align} \label{zero en}
|Q_{m,i,n+1}(z)|&\ge \left(\prod_{\delta=1}^{\lambda}(\delta-\tfrac{1}{2})\right)^m\left(\prod_{\nu=1}^{n-\lambda}(\nu-\tfrac{1}{2})\right)^{m} \left(\tfrac{1}{2}\right)^{2m} \nonumber\\
                &=\left(\dfrac{(2\lambda)!}{\lambda!2^{2\lambda}}\right)^m\left(\dfrac{(2n-2\lambda)!}{(n-\lambda)!2^{2(n-\lambda)}}\right)^m \left(\tfrac{1}{2}\right)^{2m}\nonumber\\
                &=\left({\binom{2n}{2\lambda}}^{-1}\binom{n}{\lambda}\binom{2n}{n}n!2^{-2n-2}\right)^{m}.
\end{align}
Combining the inequality (see Proof of {\cite[Theorem $1$, p.~$376$]{M2}})
\begin{align} \label{Mahler}
{\binom{2n}{2\lambda}}^{-1}\binom{n}{\lambda}\binom{2n}{n}\ge \dfrac{2^n}{n+1} \ \text{for} \ 0 \le \lambda \le n,
\end{align}
and $(\ref{zero en})$, we obtain 
\begin{align} \label{lambda conclusion}
|Q_{m,i,n+1}(z)|\ge (n+1)^{-m} 2^{-(n+2)m}n!^m  \ \text{for} \ z\in \{z\in\C| \ |z-\lambda|=\tfrac{1}{2}\}.
\end{align}
In the case of $\lambda=n+1$, using the inequality $(\ref{ineq Q})$, we obtain    
\begin{align} \label{n+1}
|Q_{m,i,n+1}(z)|\ge \left(\binom{2n+2}{n+1}(n+1)!2^{-2n-3}\right)^{m} \ \text{for} \ z\in \{z\in\C| \ |z-n-1|=\tfrac{1}{2}\}.
\end{align}
By the same arguments as above, from $(\ref{n+1})$, we obtain 
\begin{align} \label{n+1 2}
|Q_{m,i,n+1}(z)|\ge  (n+2)^{-m} 2^{-(n+2)m}(n+1)!^m\ge (n+1)^{-m} 2^{-(n+2)m}n!^m \ \text{for} \ z\in \{z\in\C| \ |z-n-1|=\tfrac{1}{2}\}.
\end{align}
Using the inequalities $(\ref{a1})$, $(\ref{lambda conclusion})$ and $(\ref{n+1 2})$, we obtain
\begin{align} \label{a conclusion}
|a_{\lambda,j}(\bold{m}_i,\boldsymbol{\omega})|\le (n+1)^{m} 2^{(n+2)m}n!^{-m},
\end{align}
for $0\le \lambda \le n+1$, $1\le j \le m$ and $1\le i \le m$. By the definition of $P_{i,j,n+1}(z)$ and use the inequalities $(\ref{a conclusion})$, we obtain 
\begin{align*}
|A_{i,j,n+1}(\alpha)|&\le (n+1)^{m} 2^{(n+2)m}n!^{-m}\sum_{h=0}^{n+1}(1+|\alpha|)^h \le \dfrac{(1+|\alpha|)^{n+2}}{|\alpha|}(n+1)^{m} 2^{(n+2)m}n!^{-m},
\end{align*}
for $\alpha\in \C$. This completes the proof of Lemma $\ref{keisu ookisa}$.
\end{proof}
\begin{lemma} \label{uekara jyouyokou}
Let $m$ be a natural number $m\ge 2$.
Let $\alpha\in \C\setminus \{0,-1\}$ satisfying $2\le m/|{\rm{log}}(1+\alpha)|$. 
Then we have
\begin{align*} 
&|R_{i,n+1}(\alpha)|\le {\rm{exp}}\left({\dfrac{2|{\rm{log}}(1+\alpha)|}{1+\sqrt{1+4|{\rm{log}}(1+\alpha)|}}}\right)\times \\
&\left[
{\rm{exp}}
\left(\dfrac{m(1+\sqrt{1+4|{\rm{log}}(1+\alpha)|})}{2}+\dfrac{2|{\rm{log}}(1+\alpha)|}{1+\sqrt{1+4|{\rm{log}}(1+\alpha)|}} \right)
\left(\dfrac{|{\rm{log}}(1+\alpha)|}{m}\right)^m
\right]^{n+1}(n+1)^{-m(n+1)}.
\end{align*}
\end{lemma}
\begin{proof}
This proof is based on that of {\cite[Theorem $1$]{M2}}.
By $(\ref{integral rep R})$, we have 
\begin{align} \label{integ rep R i n+1}
R_{i,n+1}(z)=\dfrac{1}{2\pi\sqrt{-1}}\int_{C_{\rho}}\dfrac{(1+z)^x}{(\prod_{h=0}^{n}(x-h))^m(x-n-1)^i}dx \ \text{for} \ 1 \le i \le m,
\end{align}
where $C_{\rho}$ is a circle in the $x$-plane of center $x=0$ and radius $\rho>n+1.$
In the following, we take a positive real number $\rho$ satisfying $\rho\ge 2(n+1)$.
For $x\in C_{\rho}$, we have
\begin{align}
\left|(\prod_{h=0}^{n}(x-h))^m(x-n-1)^i\right|&=\left|x^{m(n+1)+i}\left[\left(1+\dfrac{1}{x}\right)\ldots \left(1+\dfrac{n}{x}\right)\right]^{m}\left(1+\dfrac{n+1}{x}\right)^i\right| \nonumber\\
                                              &\ge \rho^{m(n+1)+i}\left[\left(1-\dfrac{1}{\rho}\right)\ldots \left(1-\dfrac{n}{\rho}\right)\right]^{m}\left(1-\dfrac{n+1}{\rho}\right)^i \nonumber \\
                                              &\ge \rho^{m(n+1)+1}\left[\left(1-\dfrac{1}{\rho}\right)\ldots \left(1-\dfrac{n+1}{\rho}\right)\right]^{m}. \label{ineq 1}
\end{align}
Since
$$\left(1-\dfrac{1}{\rho}\right)\ldots \left(1-\dfrac{n+1}{\rho}\right)=\left[\left(1+\dfrac{1}{\rho-1}\right)\ldots \left(1+\dfrac{n+1}{\rho-n-1}\right)\right]^{-1},$$
and
\begin{align*}
\left(1+\dfrac{1}{\rho-1}\right)\ldots \left(1+\dfrac{n+1}{\rho-n-1}\right)&\le {\rm{exp}}\left(\sum_{\lambda=1}^{n+1}\dfrac{\lambda}{\rho-\lambda}\right)
                                                                                                  \le {\rm{exp}}\left(\sum_{\lambda=1}^{n+1}\dfrac{\lambda}{\rho-n-1}\right) \\
                                                                                                  &= {\rm{exp}}\left(\dfrac{(n+1)(n+2)}{2(\rho-n-1)}\right) \le  {\rm{exp}}\left(\dfrac{(n+1)(n+2)}{\rho}\right),
\end{align*}
we have 
$$\left|(\prod_{h=0}^{n}(x-h))^m(x-n-1)^i\right|\ge \rho^{m(n+1)+1}  {\rm{exp}}\left(-\dfrac{m(n+1)(n+2)}{\rho}\right).$$
By $(\ref{integ rep R i n+1})$ and the above inequality, we obtain
\begin{align} \label{upper jyouyo 1}
|R_{i,n+1}(\alpha)|\le {\rm{exp}}\left(\rho|{\rm{log}}(1+\alpha)|+\dfrac{m(n+1)(n+2)}{\rho}\right) \rho^{-m(n+1)} \ \ \text{for} \ 1 \le i \le m. 
\end{align}
Put $f(x)=x|{\rm{log}}(1+\alpha)|+\dfrac{m(n+1)(n+2)}{x}-m(n+1){\rm{log}}(x)$ for $x>0$. Then $f(x)$ takes the minimal value at 
$$x=\dfrac{ m(n+1)+\sqrt{m^2(n+1)^2+4m(n+1)(n+2)|{\rm{log}}(1+\alpha)|}}{2|{\rm{log}}(1+\alpha)|}.$$
Since $n+2<m(n+1)$, we take $\rho=\dfrac{ m(n+1)(1+\sqrt{1+4|{\rm{log}}(1+\alpha)|})}{2|{\rm{log}}(1+\alpha)|}$.
Note that, by the assumption $2\le m/|{\rm{log}}(1+\alpha)|$, we have $2(n+1)\le \rho$.
By $(\ref{upper jyouyo 1})$, we obtain the desire inequality. This completes the proof of Lemma $\ref{uekara jyouyokou}$.
\end{proof}
Next, we give a $p$-adic version of Lemma $\ref{uekara jyouyokou}$.
\begin{lemma} \label{p upper bound jyouyo}
Let $\alpha\in \C_p$ satisfying $|\alpha|_p<1$. 
Then we have 
\begin{align} \label{upper bound p}
\max_{0\le i \le m-1}|d^m_{n+1}(n+1)!^m(m-1)!R_{i,n+1,p}(\alpha)|_p\le (m(n+1)+m-2)^{m-1}|\alpha|^{m(n+1)-1}_p,
\end{align}
for any natural number $n$ satisfying $1/{\rm{log}}|\alpha|^{-1}_p+1/m\le n.$
\end{lemma}
\begin{proof}
Since $R_{i,n+1}(z)$ is a wight $\bold{n}_i$ Pad\'{e} approximation of $(1,{\rm{log}}(1+z), \ldots, {\rm{log}}^{m-1}(1+z))$, we have ${\rm{ord}}R_{i,n+1}(z)=m(n+1)+i-1$.
Put $E_{n+1}:=d^m_{n+1}(n+1)!^m(m-1)!$ and $$E_{n+1} R_{i,n+1}(z)=\sum_{k=m(n+1)+i-1}r_{i,k,n+1}z^k\in \Q[[z]].$$ 
First we prove the inequalities
\begin{align} \label{abs val coeff R p}
|r_{i,k,n+1}|_p\le k^{m-1} \ \text{for} \ 1\le i \le m, m(n+1)+i-1\le k .
\end{align}
By Lemma $\ref{denominator}$, we have $E_{n+1}P_{i,j,n+1}(z)\in \Z[z]$ for $1\le i \le m$, $0\le j \le m-1$. Using the equality $$E_{n+1}R_{i,n+1}(z)=\sum_{j=0}^{m-1}E_{n+1}P_{i,j,n+1}(z){\rm{log}}^j(1+z)$$ and the definition of ${\rm{log}}(1+z)$, we have ${\rm{den}}(r_{i,k,n+1})\le k^{m-1}$ for $m(n+1)+i-1\le k$. Then we obtain the inequalities $(\ref{abs val coeff R p})$.
Using the inequalities $(\ref{abs val coeff R p})$, we have 
\begin{align*}
|E_{n+1}R_{i,n+1,p}(\alpha)|_p\le \max_{m(n+1)+i-1\le k}|r_{i,k,n+1}\alpha^{k}|_p\le \max_{m(n+1)+i-1\le k}k^{m-1} |\alpha^{k}|_p.
\end{align*}
Since we have $\max_{m(n+1)+i-1\le k}|r_{i,k,n+1}\alpha^{k}|_p\le (m(n+1)+i-1)^{m-1}|\alpha|^{m(n+1)+i-1}_p$
for any natural number $n$ satisfying $1/{\rm{log}}|\alpha|^{-1}_p+1/m\le n$, we have the desire inequalities. 
This completes the proof of Lemma $\ref{p upper bound jyouyo}.$
\end{proof}
\section{Proof of main theorems}
In this section, we give the proofs of Theorem $\ref{power of log indep}$ and Theorem $\ref{p power of log indep}$.
\subsection{Non-vanishing of certain determinants}
\begin{lemma}  $($cf. {\rm{\cite[Theorem $1.2.3$]{J}}} $)$ \label{non vanishing of det}
Let $K$ be a field with characteristic $0$ and $\bold{f}=(1,f_1,\ldots,f_{m-1})\in K[[z]]^m$.
Let $\bold{n}=(n_1,\ldots,n_m)\in \Z^{m}_{\ge 0}$. Put $\bold{n}_i=(n_1+1,n_2+1,\ldots,n_i+1,n_{i+1},\ldots,n_m) \ \text{for} \ 1\le i \le m.$
Let $(A_{i,1}(z),\ldots,A_{i,m}(z))\in K[z]^m$ be a weight $\bold{n}_i$ Pad\'{e} approximants of $\bold{f}$. We define a polynomial $\Delta(z)$ by 
\begin{equation} \label{det}
                     \Delta(z)={\begin{pmatrix}
                     A_{1,1}(z)& A_{1,2}(z) & \dots &A_{1,m}(z)\\
                     A_{2,1}(z)& A_{2,2}(z) & \dots &A_{2,m}(z)\\
                     \vdots & \vdots & \ddots  &\vdots \\
                     A_{m,1}(z)& A_{m,2}(z) & \dots &A_{m,m}(z)\\
                     \end{pmatrix}}. 
                     \end{equation}
Then there exists $\gamma\in K$ satisfying 
\begin{align} \label{calculation of det}
\Delta(z)=\gamma z^N,
\end{align} 
where $N=\sum_{j=1}^m (n_j+1)$.
Moreover, if the set of indices $\{\bold{n}\}\cup \{\bold{n}_i\}_{1\le i \le m-1}$ are normal with respect to $\bold{f}$, we have $\Delta(z)\neq 0$, i.e. $\gamma\neq 0$.
\end{lemma}
\begin{proof}
Denote the formal power series $A_{i,1}(z)+ A_{i,2}(z)f_1(z)+\dots+A_{i,m}(z)f_{m-1}(z)$ by $R_i(z)$ for $1\le i \le m$. Note that we have 
\begin{align} \label{order lower bound}
{\rm{ord}}R_i(z)\ge N+i-1 \ \text{for} \ 1\le i \le m.
\end{align}
By adding the $i$-th column of the matrix $(\ref{det})$ multiplied by $f_{i-1}(z)$ to the first column of the matrix $(\ref{det})$ for all $2\le i \le m$, we obtain the following equality:
\begin{align} \label{equal determ 1}
\Delta(z)=
                      \mathrm{det}{\begin{pmatrix}
                     R_1(z)& A_{1,2}(z) &\dots & A_{1,m}(z)\\
                     R_2(z)& A_{2,2}(z) & \dots & A_{2,m}(z)\\
                     \vdots & \vdots & \ddots & \vdots\\
                     R_m(z)& A_{m,2}(z) & \dots & A_{m,m}(z)\\
                     \end{pmatrix}}.                    
                     \end{align} 
For $1\le t,u\le m$, we denote the $(t,u)$-th cofactor of the matrix in $(\ref{equal determ 1})$ by $\Delta_{t,u}(z)$. Then, we obtain 
\begin{align} \label{decomp det} 
\Delta(z)=\displaystyle\sum_{t=1}^{m}R_{t}(z)\Delta_{t,1}(z).
\end{align}
Using the inequalities $(\ref{order lower bound})$ and the equality $(\ref{decomp det})$, we have 
\begin{align} \label{order lower bound delta}
{\rm{ord}}\Delta(z)\ge N.
\end{align}
On the other hand, by using the equality $(\ref{det})$, we obtain 
\begin{align} \label {upper bound det}
{\rm{ord}}\Delta(z)\le N.
\end{align} 
Combining the inequalities $(\ref{order lower bound delta})$ and $(\ref{upper bound det})$, we obtain the equality $(\ref{calculation of det})$.
If the set of indicies $\{\bold{n}\}\cup \{\bold{n}_i\}_{1\le i \le m}$ are normal with respect to $\bold{f}$ then, using Lemma $\ref{cor fund Pade}$, we have ${\rm{deg}}A_{i,i}(z)=n_i+1$ for $1\le i \le m$.
Then we have 
\begin{align} \label {upper bound det 2}
{\rm{ord}}\Delta(z)=N.
\end{align}
This completes the proof of Lemma $\ref{non vanishing of det}$.
\end{proof}
Using Lemma $\ref{non vanishing of det}$ for $\bold{f}:=(1,{\rm{log}}(1+z),\ldots,{\rm{log}}^{m-1}(1+z))$ and $\bold{n}=(n,\ldots,n)\in \N^m$, we obtain the following corollary.
\begin{corollary} \label{key corollary}
Let $\{A_{i,j,n+1}(z)\}_{1\le i \le m,0\le j \le m-1}\subset \Q[z]$ be the set of polynomials defined in $(\ref{coeff polynomial})$. Put $\Delta^{(n+1)}(z):={\rm{det}}(A_{i,j,n+1}(z))_{1\le i \le m,0\le j \le m-1}.$
Then there exists some $\gamma\in \Q\setminus\{0\}$ satisfying 
\begin{align} \label{non vanish}
\Delta^{(n+1)}(z)=\gamma z^{(n+1)m}.
\end{align} 
Especially, we have 
\begin{align} \label{non vanish value}
\Delta^{(n+1)}(\alpha)\neq 0 \ \text{for all} \ \alpha\in \overline{\Q} \setminus\{0\}. 
\end{align}
\end{corollary}
\subsection{Proof of Theorem $\ref{power of log indep}$}
Before starting to prove Theorem $\ref{power of log indep}$, we introduce a sufficient condition to obtain a lower bound of linear forms of complex numbers with integer coefficients.
For $\boldsymbol{\beta}:=(\beta_0,\ldots,\beta_{m})\in K^{m+1}\setminus\{\bold{0}\}$ and $\theta_0,\theta_1,\ldots,\theta_m\in \C$,
we denote $\sum_{i=0}^{m}\beta_i\theta_i$ by $\Lambda(\boldsymbol{\beta},\boldsymbol{\theta})$.
\begin{proposition} \label{critere}
Let $K$ be an algebraic number field and fix an embedding of $\sigma:K\hookrightarrow \C$. We denote the completion of $K$ by the fixed embedding $\sigma$ by $K_{\infty}$.
Let $m\in\N$ and $\theta_0:=1,\theta_1,\ldots,\theta_{m}\in \C^{*}$.
Suppose there exist a set of matrices 
$$\{(a_{i,j,n})_{0\le i,j \le m}\}_{n\in \N} \subset {\rm{GL}}_{m+1}(K)\cap {\rm{M}}_{m+1}(\mathcal{O}_K),$$
positive real numbers
\begin{align*}
&\{\mathcal{A}^{(k)}\}_{1\le k \le [K:\Q]}, \{c^{(k)}_{i}\}_{\substack{0\le i \le m \\ 1\le k \le [K:\Q]}}, \{T^{(k)}\}_{1\le k \le [K:\Q]},A,c,T,N
\end{align*}
and a function $f:\N\longrightarrow \R_{\ge0}$
satisfying 
\begin{align} 
&c^{(k)}_{0}\le \ldots \le c^{(k)}_{m} \ \text{for} \ 1\le k \le [K:\Q], \nonumber\\
&f(n)=o(n) \ \ (n\to \infty), \label{fn}
\end{align}
and 
\begin{align*}
&\max_{0\le j \le m}|a^{(k)}_{i,j,n}|\le T^{(k)}n^{c^{(k)}_{i}}e^{\mathcal{A}^{(k)}n+f(n)} \ \text{for} \ 0\le i \le m \ \text{and} \ 1\le k \le [K:\Q],\\
&\max_{0\le j \le m} |\sum_{i=0}^ma_{i,j,n}\theta_i|\le Tn^{c}e^{-A n+f(n)},
\end{align*}
for $n\ge N$. 
Put 
\begin{align*}
&\delta:=A+\mathcal{A}^{(1)}-\dfrac{m\sum_{k=1}^{[K:\Q]}\mathcal{A}^{(k)} }{[K_{\infty}:\R]}\\
&\nu:=A+\mathcal{A}^{(1)}.
\end{align*}
Suppose $\delta>0$, then the numbers $\theta_0,\ldots,\theta_{m}$ are linearly independent over $K$ and, for any $\epsilon>0$, 
there exists a constant $H_0$ depending on $\epsilon$ and the given data such that the following property holds.
For any $\boldsymbol{\beta}:=(\beta_0,\ldots,\beta_m) \in \mathcal{O}^{m+1}_K \setminus \{ \bold{0} \}$ satisfying $\mathrm{H}(\boldsymbol{\beta})\ge H_0$, then we have
\begin{align}
|\Lambda(\boldsymbol{\beta},\boldsymbol{\theta})|>{\mathrm{H}(\boldsymbol{\beta})}^{-\tfrac{[K:\Q]\nu}{[K_{\infty}:\R]\delta}-\epsilon}.
\end{align}
\end{proposition}
\begin{proof}
Since ${\rm{det}}(a_{i,j,n})_{0\le i,j \le m}\neq 0$ for all $n\in \N$, there exists $0\le I_n \le m$ satisfying 
\begin{equation} \label{det}
                     \Theta_{\boldsymbol{\beta},n}:={\rm{det}}
                    {\begin{pmatrix}
                     a_{0,0,n}& a_{0,1,n} & \dots &a_{0,m,n}\\
                     \vdots & \vdots & \ddots  &\vdots \\
                     \beta_0& \beta_1 & \dots & \beta_m\\
                     \vdots & \vdots & \ddots  &\vdots \\
                     a_{m,0,n} & a_{m,1,n} & \dots & a_{m,m,n}\\
                     \end{pmatrix}}\neq 0,
                     \end{equation}
where the vector $(\beta_0,\ldots,\beta_m)$ is placed in the $I_n$-th line of the matrix in the definition of $\Theta_{\boldsymbol{\beta},n}$.
Then by the product formula, we have
\begin{align} \label{upper infty}
1\le  |\Theta^{(1)}_{\boldsymbol{\beta},n}|^{[K_{\infty}:\R]} \times {\prod_{k}}^{\prime} |\Theta^{(k)}_{\boldsymbol{\beta},n}|,
\end{align}
where ``$ \ {}^{\prime} \ $'' in ${\prod_{k}}^{\prime}$ means $k$ runs $2\le k \le [K:\Q]$ if $K_{\infty}=\R$ and $3\le k \le [K:\Q]$ if $K_{\infty}=\C$.
In the following, we denote the $(s,t)$-th cofactor of the matrix in the definition of $\Theta_{\boldsymbol{\beta},n}$ by  $\Theta_{\boldsymbol{\beta},n,s,t}$.
First we give an upper bound of $|\Theta^{(1)}_{\boldsymbol{\beta},n}|$.
\begin{align}
&|\Theta^{(1)}_{\boldsymbol{\beta},n}|=\left|{\rm{det}}
                    {\begin{pmatrix}
                     \sum_{i=0}^m a_{i,0,n}\theta_i & a_{0,1,n} & \dots &a_{0,m,n}\\
                     \vdots & \vdots & \ddots  &\vdots \\
                     \Lambda(\boldsymbol{\beta},\boldsymbol{\theta}) & \beta_1 & \dots & \beta_m\\
                     \vdots & \vdots & \ddots  &\vdots \\
                     \sum_{i=0}^ma_{i,m,n}\theta_i & a_{m,1,n} & \dots & a_{m,m,n}\\
                     \end{pmatrix}}\right| \label{upper iota}\\
                                &=|\sum_{\substack{1\le i \le m+1 \\ i\neq I_n}} \left(\sum_{i=0}^ma^{(1)}_{i,j,n}\theta_i \right)\Theta^{(1)}_{\boldsymbol{\beta},i,1}+\Lambda(\boldsymbol{\beta}, \boldsymbol{\theta})\Theta^{(1)}_{\boldsymbol{\beta},I_n,1}| \nonumber \\
                                &\le mTn^{c}e^{-An+f(n)} m!\max\{1,|\beta_i| \} \prod_{i=2}^{m} (T^{(1)}n^{c^{(1)}_{i}}e^{\mathcal{A}^{(1)}n+f(n)})+|\Lambda(\boldsymbol{\beta}, \boldsymbol{\theta})|m!\prod_{i=1}^{m}(T^{(1)}n^{c^{(1)}_{i}}e^{\mathcal{A}^{(1)}n+f(n)}) \nonumber\\
&=m!\prod_{i=2}^{m} (T^{(1)}n^{c^{(1)}_{i}}e^{\mathcal{A}^{(1)}n+f(n)})\left(mTn^{c}e^{-An+f(n)}\max\{1, |\beta_i|\}+|\Lambda(\boldsymbol{\beta}, \boldsymbol{\theta})|T^{(1)}n^{c^{(1)}_{1}}e^{\mathcal{A}^{(1)}n+f(n)}\right). \nonumber
\end{align}
Secondly, we give an upper bound of $|\Theta^{(k)}_{\boldsymbol{\beta},n}|$ for $2\le k \le [K:\Q]$.
\begin{align}
|\Theta^{(k)}_{\boldsymbol{\beta},n}|&=\left|{\rm{det}}
                    {\begin{pmatrix}
                     a^{(k)}_{0,0,n} & a^{(k)}_{0,1,n} & \dots & a^{(k)}_{0,m,n}\\
                     \vdots & \vdots & \ddots  &\vdots \\
                     \beta^{(k)}_0 &  \beta^{(k)}_1 & \dots &  \beta^{(k)}_m\\
                     \vdots & \vdots & \ddots  &\vdots \\
                     a^{(k)}_{m,0,n}& a^{(k)}_{m,1,n} & \dots & a^{(k)}_{m,m,n}\\
                     \end{pmatrix}}\right| \label{tau part infty} \\
                                &\le (m+1)!\max\{1,|\beta^{(k)}_i|\} \prod_{1\le i \le m}\left(T^{(k)} n^{c^{(k)}_{i}}e^{\mathcal{A}^{(k)}n+f(n)}\right) \nonumber\\
                                &= (m+1)!\max\{1,|\beta^{(k)}_i|\} (T^{(k)})^m n^{\sum_{1\le i \le m}c^{(k)}_{i}}e^{m(\mathcal{A}^{(k)}n+f(n))}. \nonumber
\end{align}
Substituting the inequalities $(\ref{upper iota})$ and $(\ref{tau part infty})$ to the inequality $(\ref{upper infty})$ and taking the $\tfrac{1}{[K_{\infty}:\R]}$-th power, we obtain
\begin{align} \label{conclusion 2}
1\le C_1\mathrm{H}(\boldsymbol{\beta})^{\tfrac{[K:\Q]}{[K_{\infty}:\R]}} e^{-\delta n}+C_2 \mathrm{H}(\boldsymbol{\beta})^{\tfrac{[K:\Q]}{[K_{\infty}:\R]}}|\Lambda(\boldsymbol{\beta},\boldsymbol{\theta})|e^{(\nu-\delta)n},
\end{align}
where 
\begin{align*}
&C_1:= m!mTn^{c} e^{f(n)}\prod_{i=2}^{m} (T^{(1)}n^{c^{(1)}_{i}}e^{f(n)}) {\prod_k}^{\prime} \left[(m+1)! (T^{(k)})^m n^{\sum_{1\le i \le m}c^{(k)}_{i}}e^{mf(n)}\right]^{\tfrac{1}{[K_{\infty}:\R]}},\\
&C_2:=m!\prod_{i=1}^{m} (T^{(1)}n^{c^{(1)}_{i}}e^{f(n)}){\prod_{k}}^{\prime} \left[(m+1)! (T^{(k)})^m n^{\sum_{1\le i \le m}c^{(k)}_{i}}e^{mf(n)}\right]^{\tfrac{1}{[K_{\infty}:\R]}}.
\end{align*}
Let $\epsilon>0$ and $0<\tilde{\epsilon}<\delta$ satisfying 
\begin{align}\label{condition tilde mu}
\dfrac{\nu}{\delta}+\dfrac{\epsilon}{2}\ge \dfrac{\nu}{\delta-\tilde{\epsilon}}.
\end{align}
Write $\tilde{\delta}:=\delta-\tilde{\epsilon}$.
By the assumptions $\delta>0$ and  $(\ref{fn})$, there exists a natural number ${n^{*}}$ satisfying   
\begin{align}
&C_1 e^{-\delta n}\le e^{-\tilde{\delta}n}, \label{cond 1}\\
&C_2 e^{(\nu-\delta)n}\le e^{(\nu-\tilde{\delta})n}, \label{cond 2}
\end{align}
for all $n\ge n^{*}$.
Consequently, using $(\ref{conclusion 2})$, we obtain
\begin{align} \label{conclusion 3}
|\Lambda(\boldsymbol{\beta},\boldsymbol{\theta})|\ge \dfrac{1-e^{-\tilde{\delta}n}\cdot \mathrm{H}(\boldsymbol{\beta})^{\tfrac{[K:\Q]}{[K_{\infty}:\R]}}}{e^{(\nu-\tilde{\delta})n}\cdot \mathrm{H}(\boldsymbol{\beta})^{\tfrac{[K:\Q]}{[K_{\infty}:\R]}}} \ \ \text{for any} \ n\ge n^{*}.
\end{align}
Now for this fix $n^{*}$, we consider $H_0>1$ such that $e^{-\tilde{\delta}n^{*}}H^{\tfrac{[K:\Q]}{[K_{\infty}:\R]}}_0\ge \tfrac{1}{2}$. Then we have $e^{-\tilde{\delta}n^{*}}\mathrm{H}(\boldsymbol{\beta})^{\tfrac{[K:\Q]}{[K_{\infty}:\R]}}\ge \tfrac{1}{2}$ for all $\boldsymbol{\beta}\in \mathcal{O}^{m+1}_K\setminus \{\bold{0}\}$ satisfying $\mathrm{H}(\boldsymbol{\beta})\ge H_0$.
Take $\boldsymbol{\beta}\in \mathcal{O}^{m+1}_K$ satisfying $\mathrm{H}(\boldsymbol{\beta})\ge H_0$. Let $\tilde{n}=\tilde{n}(\mathrm{H}(\boldsymbol{\beta})) \in \N$ be the least positive integer satisfying  $e^{-\tilde{\delta}\tilde{n}}\mathrm{H}(\boldsymbol{\beta})^{\tfrac{[K:\Q]}{[K_{\infty}:\R]}}<\tfrac{1}{2}$. Note that $\tilde{n}> n^{*}$. Using inequality $(\ref{conclusion 3})$ for $\tilde{n}$, we have
\begin{align}
|\Lambda(\boldsymbol{\beta},\boldsymbol{\theta})|>\dfrac{\tfrac{1}{2}}{e^{(\nu-\tilde{\delta})\tilde{n}}\cdot \mathrm{H}(\boldsymbol{\beta})^{\tfrac{[K:\Q]}{[K_{\infty}:\R]}}}.
\end{align}
By the definition of $\tilde{n}$, we have $e^{-(\tilde{n}-1)\tilde{\delta}} \mathrm{H}(\boldsymbol{\beta})^{\tfrac{[K:\Q]}{[K_{\infty}:\R]}}\ge \tfrac{1}{2}$ and then $e^{\tilde{n}}\le (2\mathrm{H}(\boldsymbol{\beta})^{\tfrac{[K:\Q]}{[K_{\infty}:\R]}})^{\tfrac{1}{\tilde{\delta}}}e$.
Finally, we obtain
\begin{align*}
|\Lambda(\boldsymbol{\beta},\boldsymbol{\theta})|&>\dfrac{1}{2^{\tfrac{\nu}{\tilde{\delta}}} \cdot e^{\nu-\tilde{\delta}} \cdot \mathrm{H}(\boldsymbol{\beta})^{\tfrac{[K:\Q]\nu}{[K_{\infty}:\R]\tilde{\delta}}}}\\
&\ge \dfrac{1}{\mathrm{H}(\boldsymbol{\beta})^{\tfrac{[K:\Q]\nu}{[K_{\infty}:\R]\tilde{\delta}}}}\\
&\ge \dfrac{1}{\mathrm{H}(\boldsymbol{\beta})^{\tfrac{[K:\Q]\nu}{[K_{\infty}:\R]\delta}+\tfrac{\epsilon[K:\Q]}{2[K_{\infty}:\R]}}}.
\end{align*}
Note that the last inequality is obtained by the inequality $(\ref{condition tilde mu})$.
This completes the proof of Proposition $\ref{critere}$.
\end{proof}
\begin{remark} \label{how to calculate}
In this remark, we explain that how to take the positive number $H_0$ in Proposition  $\ref{critere}$.
Let $\epsilon>0$. 
At first we take $\tilde{\epsilon}:=\dfrac{\epsilon \delta^2}{(2\nu+\epsilon \delta)}$. Then we have
\begin{align*}
\dfrac{\nu}{\delta}+\dfrac{\epsilon}{2}=\dfrac{\nu}{\delta-\tilde{\epsilon}}.
\end{align*} 
Since $f(n)=o(n) (n\to \infty)$, there exists $\tilde{n}^{*}\in \N$ satisfying
\begin{align*}
f(n)<\dfrac{[K_{\infty}:\R]}{2m[K:\Q]}\tilde{\epsilon} n \ \text{for any} \ n\ge \tilde{n}^{*}.
\end{align*} 
We take $n^{*}\ge \tilde{n}^{*}$ satisfying 
\begin{align}
&{\rm{log}}
\left(m!mTn^{c}\prod_{i=1}^m(T^{(1)}n^{c^{(1)}_{i}})\times 
{\prod_{k}}^{\prime} \left[(m+1)!(T^{(k)})^m n^{\sum_{1\le i \le m}c^{(k)}_{i}}\right]^{\tfrac{1}{[K_{\infty}:\R]}} \right)\le \dfrac{\epsilon \delta^2n^{*}}{4(2\nu+\epsilon \delta)}, \label{important}  
\end{align}
then we have the inequalities $(\ref{cond 1})$ and $(\ref{cond 2})$ for any $n\ge n^{*}$.
At last, we take $H_0$ by 
$$H_0=\left(\dfrac{1}{2}{\rm{exp}}[{\delta n^{*}}]\right)^{\tfrac{[K_{\infty}:\R]}{[K:\Q]}}.$$
Then by the proof of Proposition  $\ref{critere}$, 
the positive number $H_0$ satisfies the following property$:$
\begin{align*}
|\Lambda(\boldsymbol{\beta},\boldsymbol{\theta})|>{\mathrm{H}(\boldsymbol{\beta})}^{-\tfrac{[K:\Q]\nu}{[K_{\infty}:\R]\delta}-\tfrac{\epsilon[K:\Q]}{2[K_{\infty}:\R]}},
\end{align*}
for any $\boldsymbol{\beta}:=(\beta_0,\ldots, \beta_m) \in \mathcal{O}^{m+1}_K \setminus \{ \bold{0} \}$ satisfying $\mathrm{H}(\boldsymbol{\beta})\ge H_0$.
\end{remark}
\begin{proof} [Proof of Theorem $\ref{power of log indep}$]
We use the same notations as in Section $5$. Let $R_{i,n+1}(z)$ be the formal power series defined in $(\ref{exp pade S})$ for $1\le i \le m$ and $n\in\N$. 
Let $K$ be an algebraic number field. We fix an element $\alpha\in K\setminus\{0,-1\}$ satisfying the assumption in Theorem $\ref{power of log indep}$. Then we have 
\begin{align*}
R_{i,n+1}(\alpha)=\sum_{j=0}^{m-1}A_{i,j,n+1}(\alpha)\log^j(1+\alpha).
\end{align*}
Put 
\begin{align}
&D_{n+1}(\alpha):=d^m_{n+1}(n+1)!^m(m-1)!{\rm{den}}^{n+1}(\alpha), \label{D_n}\\
&a_{i,j,n+1}(\alpha):=D_{n+1}(\alpha) A_{i,j,n+1}(\alpha) \ \text{for} \ 1\le i \le m, 0\le j\le m-1. \label{linear form coefficients}
\end{align}
Then by Lemma $\ref{denominator}$ and Corollary $\ref{key corollary}$, we have 
\begin{align} \label{non-vanish log}
(a_{i,j,n+1}(\alpha))_{1\le i \le m,0\le j \le m-1}\ \in {\rm{GL}}_{m}(K)\cap {\rm{M}}_m(\mathcal{O}_K) \ \text{for all} \ n\in\N.
\end{align}
Define the set of positive real numbers
\begin{align*}
&\mathcal{A}^{(k)}(\alpha)=m(1+{\rm{log}}(2))+{\rm{log}}({\rm{den}}(\alpha))+{\rm{log}}(1+|\alpha^{(k)}|) \ \text{for} \ 1\le k \le [K:\Q],\\
&c^{(k)}_{l}=2m \ \text{for} \ 0\le l \le m-1 \ \text{and} \ 1\le k \le [K:\Q],\\
&T^{(k)}(\alpha)=\dfrac{2^m(1+|\alpha^{(k)}|)(m-1)!}{|\alpha^{(k)}|} \ \text{for} \ 1\le k \le [K:\Q], \\
&A(\alpha)=\dfrac{m}{2}{\rm{log}}\left(\dfrac{m}{|{\rm{log}}(1+\alpha)|}\right)-\left(\dfrac{m(1+\sqrt{1+4|{\rm{log}}(1+\alpha)|})}{2}+\dfrac{2|{\rm{log}}(1+\alpha)|}{1+\sqrt{1+4|{\rm{log}}(1+\alpha)|}} \right)-{\rm{log}}({\rm{den}}(\alpha)), \\
&c(\alpha)=\dfrac{m}{2}, \\
&T(\alpha)={\rm{exp}}\left({\dfrac{2|{\rm{log}}(1+\alpha)|}{1+\sqrt{1+4|{\rm{log}}(1+\alpha)|}}}\right)(m-1)!.
\end{align*}
Define $g(n):=n\left[\sqrt{{\rm{log}}n}\cdot {\rm{exp}}\left(-\sqrt{({\rm{log}}n)/R}\right)\right]$ with $R:=\tfrac{515}{(\sqrt{546}-\sqrt{322})^2}$.
Note that, in \cite{R-S1}, Rosser-Schoenfeld gives an estimate of $d_n$ of the form 
\begin{align} \label{growth dn}
{\rm{exp}}(n-g(n))\le d_n \le {\rm{exp}}(n+g(n)).
\end{align}
Put $f(n):=mg(n)$. 
Then we see $f(n)=o(n) \ (n\to \infty)$.
By Lemma $\ref{keisu ookisa}$ and Lemma $\ref{uekara jyouyokou}$, we have
\begin{align*}
&\max_{\substack{1\le i \le m \\ 0\le j \le m-1}}|a^{(k)}_{i,j,n+1}(\alpha)|\le T^{(k)}(\alpha) (n+1)^{c^{(k)}_{i}}e^{\mathcal{A}^{(k)}(\alpha)(n+1)+f(n+1)} \ \text{for} \ 0\le i \le m \ \text{and} \ 1\le k \le [K:\Q],\\
&\max_{1\le i \le m} \left|\sum_{j=0}^{m-1}a_{i,j,n+1}(\alpha){\rm{log}}^j(1+\alpha)\right|\le T(\alpha)(n+1)^{c(\alpha)}e^{-A(\alpha)(n+1)+f(n+1)},
\end{align*}
for all $n\in \N$. 
We use Proposition $\ref{critere}$ for $\theta_1={\rm{log}}(1+\alpha),\ldots,\theta_{m-1}={\rm{log}}^{m-1}(1+\alpha)$ 
and the above datum 
$\{\mathcal{A}^{(k)}(\alpha)\}_{1\le k \le [K:\Q]}$, 
$\{c^{(k)}_{l}\}_{\substack{0\le l \le m-1 \\ 1\le k \le [K:\Q]}}$,
$\{T^{(k)}(\alpha)\}_{1\le k \le [K:\Q]}$,
$A(\alpha)$,
$c(\alpha)$, 
and $T(\alpha)$ then we obtain the assertion of Theorem $\ref{power of log indep}$.
\end{proof}
\subsection{Proof of Theorem $\ref{p power of log indep}$}
Before proving Theorem  $\ref{p power of log indep}$, we introduce a $p$-adic version of Proposition $\ref{critere}$.
\begin{proposition} \label{critere p}
Let $K$ be an algebraic number field and fix an embedding $\iota_p:K\longrightarrow \C_p$. 
We denote the completion of $K$ by the fixed embedding $\sigma_{p}$ by $K_{p}$.
Let $m\in\N$ and $\theta_0:=1,\theta_1,\ldots,\theta_{m-1}\in \C_p$.

Suppose, for all $n\in\N$, there exist a set of matrices 
$$\{(a_{i,j,n})_{0\le i,j \le m}\}_{n\in \N} \subset {\rm{GL}}_{m+1}(K)\cap {\rm{M}}_{m+1}(\mathcal{O}_K),$$
positive real numbers
\begin{align*}
&\{\mathcal{A}^{(k)}\}_{1\le k \le [K:\Q]}, \{c^{(k)}_{i}\}_{\substack{0\le i \le m \\ 1\le k \le [K:\Q]}}, \{T^{(k)}\}_{1\le k \le [K:\Q]},A_p,c_p,T_p,N
\end{align*}
and a function $f:\N\longrightarrow \R_{\ge0}$ satisfying 
\begin{align*}
&c^{(k)}_{0}\le \ldots \le c^{(k)}_{m} \ \text{for} \ 1\le k \le [K:\Q], \nonumber\\
&f(n)=o(n) \ \ (n\to \infty),
\end{align*}
and
\begin{align*}
&\max_{0\le j \le m}|a^{(k)}_{i,j,n}|\le T^{(k)}n^{c^{(k)}_{i}}e^{\mathcal{A}^{(k)}n+f(n)} \ \text{for} \ 0\le i \le m \ \text{and} \ 1\le k \le [K:\Q],\\
&\max_{0\le j \le m} |\sum_{i=0}^{m}a_{i,j,n}\theta_i|_p\le Tn^{c_p}e^{-A_pn},
\end{align*}
for all $n\ge N$. Put
\begin{align*}
&\delta_p:=A_p-\dfrac{m\sum_{k=1}^{[K:\Q]}\mathcal{A}^{(k)}}{[K_p:\Q_p]},\\
&\nu_p:=A_p.
\end{align*}
For $\epsilon>0$ and ${n}^{*}\in \N$ satisfying $n^{*}\ge N$ and 
\begin{align*}
&f(n)\le \dfrac{\epsilon \delta^2_p[K_{p}:\Q_p]n^{*}}{2m(2\nu_p+\epsilon \delta_p)[K:\Q]}, \\
&{\rm{log}}\left(T_{p}n^{b_p} \left[\prod_{k=1}^{[K:\Q]}(m+1)!(T^{(k)})^m n^{\sum_{1\le i \le m}c^{(k)}_{i}}\right]^{\tfrac{1}{[K_{p}:\Q_p]}}\right) \le \dfrac{\epsilon \delta^2_pn^{*}}{4 (2\nu_p+\epsilon \delta_p)},
\end{align*}
for all $n\ge n^{*}$.
Suppose $\delta_p>0$, then the numbers $\theta_0,\ldots,\theta_{m}$ are linearly independent over $K$ and the positive number
$H_0:=\left(\dfrac{1}{2}{\rm{exp}}\left[\delta_p n^{*}\right]\right)^{\tfrac{[K_p:\Q_p]}{[K:\Q]}}$
satisfies the following property$:$

For any $\boldsymbol{\beta}:=(\beta_0,\ldots, \beta_m) \in \mathcal{O}^{m+1}_K \setminus \{ \bold{0} \}$ satisfying $H_0\le \mathrm{H}(\boldsymbol{\beta})$, then we have
\begin{align*}
|\Lambda_p(\boldsymbol{\beta},\boldsymbol{\theta})|_p>\mathrm{H}(\boldsymbol{\beta})^{-\tfrac{[K:\Q]\nu_p}{[K_{p}:\Q_p]\delta_p}-\tfrac{\epsilon[K:\Q]}{2[K_{p}:\Q_p]}} .
\end{align*}
\end{proposition}

Since Proposition $\ref{critere p}$ can be proved by the same argument of that of Proposition $\ref{critere}$, we omit the proof.

\bigskip

\begin{proof} [Proof of Theorem $\ref{p power of log indep}$]
We use the same notations as in the proof of Theorem $\ref{power of log indep}$.  Let $K$ be an algebraic number field. We fix an element $\alpha\in K\setminus\{0,-1\}$ satisfying the assumption in Theorem $\ref{p power of log indep}$. 
Put 
\begin{align*}
&T_p(\alpha)=\dfrac{(2m)^{m-1}}{|\alpha|_p},\\
&c_p=m-1,\\
&A_p(\alpha)=-m{\rm{log}}(|\alpha|_p).
\end{align*}
By Lemma $\ref{keisu ookisa}$ and Lemma $\ref{p upper bound jyouyo}$, we obtain 
\begin{align*}
&\max_{\substack{1\le i \le m \\ 0\le j \le m-1}}|a^{(k)}_{i,j,n+1}(\alpha)|\le T^{(k)}(\alpha) (n+1)^{c^{(k)}_{i}}e^{\mathcal{A}^{(k)}(\alpha)(n+1)+f(n+1)},\\ 
&\max_{1\le i \le m} \left|\sum_{j=0}^{m-1} a_{i,j,n+1}(\alpha)\log^{j}_p(1+\alpha))\right|_p\le T_p(\alpha) (n+1)^{c_p} e^{-A_p(\alpha)(n+1)},
\end{align*}
for all natural number $n$ satisfying $1/{\rm{log}}(|\alpha|^{-1}_p)+1/m\le n$.
Using Proposition $\ref{critere p}$ for 
$\theta_1={\rm{log}}_p(1+\alpha),\ldots,\theta_{m-1}={\rm{log}}^{m-1}_p(1+\alpha)$, we obtain the assertion of Theorem $\ref{p power of log indep}$.
\end{proof}

Acknowledgements. The author warmly thank Noriko Hirata-Khono for her comments on the earlier version of this manuscript.

\bibliography{}

\begin{thebibliography}{99}%
\bibitem{B}
A.~Baker,
\emph{Approximations to the logarithms of certain rational numbers},
Acta Arith.\ \textbf{10} (1964), p.~315--323.

\bibitem{C}
P.~L.~Cijsouw,
\emph{Transcendence measures of exponentials and logarithms of algebraic numbers},
Compositio Math. \textbf{28} (1974) , pp.~163--178.

\bibitem{F}
N.~I.~Fel'dman,
\emph{Approximation of the logarithms of algebraic numbers by algebraic numbers},
English transl., Amer.\ Math.\ Soc.\ Transl.\ II ser., \textbf{58} (1966), pp.~125--142.

\bibitem{G}
A.~O.~Gel'fond,
\emph{Transcendental and algebraic numbers},
English transl.\ (Dover Publications, New York (1960)).

\bibitem{H}
C.~Hermite,
\emph{Sur la fonction exponentielle},
Oeuvres tome III (1873), p.~150--181.

\bibitem{J}
H.~Jager,
\emph{A multidimensional generalization of the Pad\'{e} table. I, I\hspace{-.1em}I, I\hspace{-.1em}I\hspace{-.1em}I, I\hspace{-.1em}V, V, V\hspace{-.1em}I},
Nederl.~Ak.~Wetenschappen, \textbf{67} (1964), p.~192--249.


\bibitem{M1}
K.~Mahler,
\emph{Zur Approximation der Exponential funktion und des Logarithmus. Teil I\hspace{-.1em}I},
J.\ Reine.\ Angew.\ Math., \textbf{166} (1932), p.~137--150.

\bibitem{M2}
K.~Mahler,
\emph{On the approximation of logarithms of algebraic numbers},
Phil.\ Trans.\ Royal Soc., \textbf{245} (1953), p.~371--398.

\bibitem{M3}
K.~Mahler,
\emph{Application of some formulas by Hermite to the approximation of exponentials and logarithms},
Math.\ Ann., \textbf{168} (1967), p.~200--227.

\bibitem{M4}
K.~Mahler,
\emph{Perfect systems},
Compositio Mathematica (1968) Volume: \textbf{19}, Issue: 2, p.~95--166.

\bibitem{Mig}
M.~Mignotte, 
\emph{Approximations rationnlles de $\pi$ et quelques autres nombres},
Journ\'{e}es Arithm\'{e}tiques (Grenoble, 1973), Soc. Math. France, Paris (1974) Bull. Soc. Math. France, M\'{e}m. \textbf{37}. p.~121--132. 

\bibitem{N-W}
Yu.~Nesterenko, M.~Waldschmidt,
\emph{On the approximation of the values of exponential function and logarithm by algebraic numbers}
Mat. Zapiski, \textbf{2}, Diophantine approximations, Proceedings of papers dedicated to the memory of Prof. N.I.~Feldman, ed. Yu.V.~Nesterenko, Centre for applied research under Mech.-Math.
Faculty of MSU, Moscow (1996), pp.~23--42.

\bibitem{N-S}
E.~M.~Nikisin, V.~N.~Sorokin,
\emph{Rational Approximations and Orhogonality $($Translations of Mathematical Monographs$)$},
American Mathematical Society, (1991). 

\bibitem{R}
E.~Reyssat,
\emph{Mesures de transcendence pour les logarithmes de nombres rationnels},
Approximations diophantiennes et nombres transcendants, Luminy, 1982 Progress in Math, p.~235--245, Birkh$\ddot{\text{a}}$user (1983).

\bibitem{R-S1}
J.~B.~Rosser, L.~Schoenfeld,
\emph{Approximate formulas for some functions of prime numbers},
Illinois J.\ Math.\ \textbf{6}, p.~64--94, (1962).

\bibitem{R-S2}
J.~B.~Rosser, L.~Schoenfeld,
\emph{Shaper bounds for the Chebyshev functions $\theta(x)$ and $\psi(x)$},
Math.\ Comp.\ \textbf{29}, p.~243--269, (1975).

\bibitem{S}
C.~Siegel,
\emph{Uber einige Anwendungen diophantischer Approximationen}, 
Abhandlungen der Preu$\beta$ischen Akademie der Wissenschaften.
Physikalisch-mathematische Kalasse 1929, Nr.1.

\bibitem{W1}
M.~Waldschmidt,
\emph{Transcendence measures for exponentials and logarithms},
J.\ Austral.\ Math.\ Soc. (Series A) \textbf{25} (1978), pp.~445--465.

\bibitem{W2}
M.~Waldschmidt,
\emph{Diophantine approximation on linear algebraic groups: Transcendence properties of the exponential function in several variables},
Grundlehren der mathematischen Wissenschaften, \textbf{326}. Springer-Verlag, Berlin Heidelberg, 2000.

\bibitem{Wei}
K.~Weierstrass,
\emph{Zu Lindemann's Abhandlung ``$\ddot{u}$ber die Ludolph'sche Zahl''},
Sitzungsberichte der K$\ddot{\text{o}}$niglich-Preu$\beta$ischen Akademie der Wissenschaften (1885), p.~1067--1085.

\end{thebibliography}

\medskip\vglue5pt
\vskip 0pt plus 1fill
\hbox{\vbox{\hbox{Makoto \textsc{Kawashima}}
\hbox{Mathematical Institute,}
\hbox{Osaka University,}
\hbox{Machikaneyama, Toyonaka,}
\hbox{Osaka, 560-8502, Japan}
\hbox{{\tt m-kawashima@math.sci.osaka-u.ac.jp}}
}}

\end{document}